\documentclass[11pt,reqno]{amsart}%
\usepackage[a4paper,hscale=0.7,vscale=0.75,centering]{geometry}
\usepackage{amssymb}
\usepackage{color}
\usepackage{amsmath}
\usepackage{amsfonts}
\usepackage{graphicx}%
\renewcommand{\P}{\mathrm{P}}
\newtheorem{prop}{Proposition}[section]
\newtheorem{thm}[prop]{Theorem}
\newtheorem{cor}[prop]{Corollary}
\newtheorem{lemma}[prop]{Lemma}

\theoremstyle{definition}
\newtheorem{hyp}[prop]{Hypothesis}
\theoremstyle{remark}
\newtheorem{rmk}[prop]{Remark}
\numberwithin{equation}{section}

\begin{document}
\title{Hitting probabilities for general \\Gaussian processes}
\author{Eulalia Nualart}
\address[Eulalia Nualart]{Department of Economics and Business, Universitat Pompeu Fabra
and Barcelona Graduate School of Economics, Ram\'on Trias Fargas 25-27, 08005
Barcelona, Spain}
\email{eulalia@nualart.es}
\urladdr{http://nualart.es}
\author{Frederi Viens}
\address[Frederi Viens]{Department of Statistics, Purdue University, West Lafayette, IN
47907-2067, USA}
\email{viens@purdue.edu}
\urladdr{http://www.stat.purdue.edu/$\sim$viens}
\thanks{First author acknowledges support from
the European Union programme FP7-PEOPLE-2012-CIG under grant agreement 333938.
Second author partially supported by NSF grant DMS 0907321.}

\begin{abstract}
For a scalar Gaussian process $B$ on $\mathbb{R}_{+}$ with a prescribed
general variance function $\gamma^{2}\left(  r\right)  =\mathrm{Var}\left(
B\left(  r\right)  \right)  $ and a canonical metric $\mathrm{E}[\left(
B\left(  t\right)  -B\left(  s\right)  \right)  ^{2}]$ which is commensurate
with $\gamma^{2}\left(  t-s\right)  $, we estimate the probability for a
vector of $d$ iid copies of $B$ to hit a bounded set $A$ in $\mathbb{R}^{d}$,
with conditions on $\gamma$ which place no restrictions of power type or of
approximate self-similarity, assuming only that $\gamma$ is continuous,
increasing, and concave, with $\gamma\left(  0\right)  =0$ and $\gamma
^{\prime}\left(  0+\right)  =+\infty$. We identify optimal base (kernel)
functions which depend explicitly on $\gamma$, to derive upper and lower
bounds on the hitting probability in terms of the corresponding generalized
Hausdorff measure and non-Newtonian capacity of $A$ respectively. The proofs
borrow and extend some recent progress for hitting probabilities estimation,
including the notion of two-point local-nondeterminism in Bierm\'{e}, Lacaux,
and Xiao \cite{Bierme:09}. These techniques are part of a well-known strategy,
used in various contexts since the 1970's in the study of fine path
properties, of using covering arguments for upper bounds, and
second-moment-based energy estimates for lower bounds. Other techniques, such
as a reliance on classical Gaussian path regularity theory, or quantitative
estimates based on H\"{o}lder continuity or indexes, must be entirely
abandonned because they cannot provide results which are sharp enough.
Instead, all calculations are intrinsic to $\gamma$, and we use new density
estimation techniques based on the Malliavin calculus in order to handle the
probabilities for scalar processes to hit points and small balls. We apply our
results to the probabilities of hitting singletons and fractals in
$\mathbb{R}^{d}$, for a two-parameter class of processes. This class is fine
enough to narrow down where a phase transition to point polarity (zero
probability of hitting singletons) might occur. Previously, the transition
between non-polar and polar singletons had been described as the single point
where a process is $H$-H\"{o}lder-continuous in the mean-square with $H=1/d$;
now we can see how a range of logarithmic corrections affects this transition.

\end{abstract}
\date{March 2014}
\maketitle

\section{Introduction}

\subsection{Background and motivation}

In this paper we will assume throughout that $B=(B(t),t\in\mathbb{R}_{+})$ is
a centered continuous Gaussian process in $\mathbb{R}$ such that for some
constant $\ell\geq1$, some continuous strictly increasing function
$\gamma:\mathbb{R}_{+}\rightarrow\mathbb{R}_{+}$ with $\lim_{0}\gamma=0$, and
for all $s,t\in\mathbb{R}_{+}$,
\begin{equation}
(1/\ell)\gamma^{2}(|t-s|)\leq\mathrm{E}[|B(t)-B(s)|^{2}]\leq\ell\gamma
^{2}(|t-s|). \label{e1}%
\end{equation}
The above condition yields, with $s=0$, that $\mathrm{Var}B\left(  t\right)  $
is commensurate with $\gamma^{2}(t)$; in addition, we also assume throughout
the paper that $\gamma$ and $\mathrm{Var}B\left(  \cdot\right)  $ actually
coincide: for all $t\in\mathbb{R}_{+}$,
\begin{equation}
\mathrm{Var}B\left(  t\right)  =\gamma^{2}\left(  t\right)  . \label{e1v}%
\end{equation}
Note that $\gamma$ does not define the law of $B$ since distinct processes
with the same variance function $\gamma$ may satisfy (\ref{e1}). For economy
of notation, we also often use the letter $B$ to designate a vector of $d$ iid
copies of the scalar version of $B$. Whether $B$ needs to represent a scalar
or vector version should be clear from the context.

A number of Gaussian processes satisfy (\ref{e1}) and (\ref{e1v}), including
each fractional Brownian motion (fBm) with Hurst parameter $H\in(0,1)$, in
which case, $\ell=1$ and $\gamma\left(  t\right)  =t^{H}$ : in other words,
the inequalities in (\ref{e1}) are equalities, and they reflect fBm's
self-similarity and stationarity; Brownian motion $W$ is included therein, by
taking $H=1/2\,$. The so-called Riemann-Liouville fractional Brownian motion
(RL-fBm) with parameter $H$, defined via standard Brownian motion $W$ by
$B^{RL,H}\left(  t\right)  :=\sqrt{2H}\int_{0}^{t}\left(  t-s\right)
^{H-1/2}dW\left(  s\right)  $ is also covered, with $\gamma\left(  t\right)
=t^{H}$ just like with fBm, but with $\ell=2$: this process has non-stationary
increments, see \cite{Mocioalca:04}. The fBm and the RL-fBm are both
self-similar with parameter $H$, both correspond to $\gamma\left(  t\right)
=t^{H}$, and therefore exhibit continuity properties that have traditionally
been evaluated within the H\"{o}lder scale. This pattern of concentrating on
the value of $H$ has permeated research on the estimation of hitting
probability properties, see for e.g. \cite{Bierme:09, Testard:86, Xiao:99,
Xiao:09}. Another important class of self-similar examples with non-stationary
increments, which still satisfy (\ref{e1}) and (\ref{e1v}), are the solutions
of the stochastic heat equation with additive noise whose space behavior is of
Riesz-kernel type, as described in the preprint to appear \cite{OT}, and the
preprint \cite{TTV-heat}. Solutions of stochastic heat equations can also
easily lose the self-similarity property: see those studied in
\cite{Nualart:09} in the context of hitting probabilities, where the
H\"{o}lder scale is still the dominant yardstick; in those examples,
(\ref{e1}) and (\ref{e1v}) are still satisfied, for a function $\gamma\left(
t\right)  $ which is equivalent to $t^{H}$ as $t\rightarrow0$ for some
$H\in(0,1)$, hence the H\"{o}lder property. Any process satisfying (\ref{e1})
and (\ref{e1v}) where $\gamma\left(  t\right)  $ is commensurate with $t^{H}$
for small $t$ will still live in this H\"{o}lder scale.

One motivation of the present paper is to avoid such restrictions, by letting
the function $\gamma$, whether it be commensurate with a power function or
not, tell us how hitting probabilities behave. Classical results of R. Dudley
and others from Gaussian continuity theory (see \cite{Adler:90}) tell us that
under (\ref{e1}), if $\gamma\left(  r\right)  =o\left(  \log^{-1/2}\left(
1/r\right)  \right)  $ for $r$ near zero, then $B$ is almost-surely
continuous, and the function $h:r\mapsto\gamma\left(  r\right)  \log
^{1/2}\left(  1/r\right)  $ is, up to a deterministic constant, a uniform
modulus of continuity for $B$; i.e.%
\begin{equation}
\sup_{0\leq s<t<T}\frac{|B(t)-B(s)|}{h(t-s)}<\infty\qquad\text{a.s.}
\label{mc}%
\end{equation}
Such precise quantitative results depending only on a general $\gamma$ should
also be available for hitting probability questions, and this is the question
we try to address in this paper. In the case of moduli of continuity for
Gaussian processes, while these results are known to be sharp in some cases
(see a rather general treatment for stationary processes on the unit circle in
\cite{Tindel:04}), there are generally no lower-bound results based on
$\gamma$. In this article we will try to go one step further for hitting
probabilities, as we will derive upper \emph{and} lower bounds which both
relate to the function $\gamma$.

The basic structure of the argument in order to achieve a goal of estimating
hitting probabilities is now classical: to apply a covering argument in one
direction to get a Hausdorff-measure upper bound, and an energy estimate
(second moment argument with Paley-Zygmund inequality) in the other to get a
capacity lower bound. We will use that strategy. Its tools have been developed
and applied over the years by Albin \cite{Albin:92}, Hawkes \cite{Hawkes:77},
Kahane \cite{Kahane:82, Kahane:85, Kahane:85a}, Monrad, Pitt \cite{Monrad:87},
Testard \cite{Testard:86}, Weber \cite{Weber:83}, Xiao \cite{Xiao:99}, and
many others. In our context, setting up the main ingredients needed to apply
the strategy have led us to developing new ideas. For instance, we use a
promising techniques from the Malliavin calculus to bound the density of the
distance between a path of $B$ and a fixed point; this technique could also be
useful to other solve potential-probabilistic problems in difficult contexts
such as some of the critical non-Markovian ones. We also find a new sharp
estimate of that distance's atom at the origin in one dimension, which is the
probability for a component of $B$ to hit a point, a result of independent
interest. For capacity lower bounds, we derive the intrinsic form of the
potential kernel needed to obtain optimal bounds, as a function of $\gamma$,
with no reference to power scales or indexes; this new formula could inform
the question of how sets transition between being polar and non-polar.

These original new techniques inscribe themselves in a long tradition of
studying fine path properties of stochastic processes, illustrated by the
references listed above. While none of these references can be used directly
to establish the new tools we need in this paper, we describe each reference
here briefly to illustrate the similarities of purposes and of strategies.
Albin \cite{Albin:92} establishes a law of the iterated logarithm for a
general class of real valued stationary process, with no power-scale
restriction. The lower bound proof uses a second moment argument using the
Paley-Zygmund inequality, as often done for capacity lower bounds, including
ours. The upper bound uses a covering argument, similar to one of the keys in
proving Hausdorff measure upper bounds, as we do. In a general-theory context,
Hawkes \cite{Hawkes:77} proposes the idea of using capacity and Hausdorff
measures defined with respect to a general kernel and function; doing so is a
key requirement for our study. Then Hawkes studies the Hausdorff dimension of
the level sets and graph for a Gaussian process with stationary increments; it
is computed in terms of the variance of the increments, again using a second
moment argument for the lower bound and a covering argument for the upper
bound. While the motivation of Hausdorff dimension leads to studying power
scales and indexes, this paper is a precursor for our context. Other works
that use this dual approach of second-moment energy estimate plus covering
argument include: Kahane \cite{Kahane:82}, who studied the multiple points of
stable symmetric L\'{e}vy processes; Testard \cite{Testard:86} who studied the
existence of $k$-multiple points for the $N$-parameter $d$-dimensional fBm;
Weber \cite{Weber:83} who obtained the Hausdorff dimension of the $k$-multiple
times for these fBm, and Xiao \cite{Xiao:99} who proved upper and lower bounds
for the hitting probabilities for these same processes. This last work was
generalized by Bierm\'{e}, Lacaux, and Xiao \cite{Bierme:09} to general
Gaussian processes in the power scale. For a similar class of Gaussian
processes in the power scale, see Kahane \cite{Kahane:85a} and Monrad, Pitt
\cite{Monrad:87}, who compute the Hausdorff dimension of the range, graph, and
level sets.

To illustrate what kinds of Gaussian processes are covered by our study,
consider the following. One extends the RL-fBm via the class of so-called
Gaussian Volterra processes defined as the family containing one process
$B^{\gamma}$ for each function $\gamma$, via a standard Brownian motion $W$
and the formula%
\begin{equation}
B^{\gamma}\left(  t\right)  :=\int_{0}^{t}\sqrt{\left(  \frac{d\gamma^{2}}%
{dt}\right)  \left(  t-s\right)  }dW\left(  s\right)  . \label{VolterraBgamma}%
\end{equation}
If $\gamma^{2}$ is of class $\mathcal{C}^{2}$ on $\mathbb{R}_{+}\setminus0$,
$\lim_{0}\gamma=0$, and $\gamma^{2}$ is increasing and concave ($d\gamma
^{2}/dr$ is non-increasing), then $B^{\gamma}$ satisfies (\ref{e1}) and
(\ref{e1v}) for the fixed function $\gamma$, with $\ell=2$. See
\cite{Mocioalca:04} and \cite{Mocioalca:07}. In addition, it is a simple
matter to produce processes that satisfy the same conditions (\ref{e1}) and
(\ref{e1v}), with possibly different constants $\ell$: in the definition of
$B^{\gamma}$ in (\ref{VolterraBgamma}), replace $\left(  d\gamma
^{2}/dt\right)  \left(  t-s\right)  $ by any function of the pair $\left(
s,t\right)  $ which is bounded above and below by multiples of $\left(
d\gamma^{2}/dt\right)  \left(  t-s\right)  $. All of these processes have the
added bonus that they are adapted to a Brownian filtration. None of them have
stationary increments. When $\gamma$ is not commensurate with a power function
near $0$, the resulting $B^{\gamma}$ is far from being self-similar.

\subsection{Summary of results, and comments}

This paper is devoted to proving upper and lower bounds for the probability
that a $d$-dimensional process $B$ with iid coordinates satisfying conditions
(\ref{e1}) and (\ref{e1v}) hits a given Borel set $A$ in $\mathbb{R}^{d}$, in
terms of certain $\gamma$-dependent Hausdorff measure and capacity of the set
$A$, respectively. In this subsection, we provide a summary of our main results.

In \textbf{Section 2} we show (Theorem \ref{tcap}) that if the function
$\gamma$ in (\ref{e1}) is strictly concave in a neighborhood of zero, then for
all $0<a<b$ and $M>0$, there exists a constant $C>0$ depending only on $a,b,M$
and the law of $B$, such that for any Borel set $A\subset\lbrack-M,M]^{d}$
\begin{equation}
C\mathcal{C}_{\mathrm{K}}(A)\leq\P (B([a,b])\cap A\neq\varnothing
),\label{lowerc}%
\end{equation}
where $\mathcal{C}_{\mathrm{K}}(A)$ denotes the capacity of the set $A$ with
respect to the potential kernel
\begin{equation}
\mathrm{K}(x):=\max\left\{  1;v\left(  \gamma^{-1}\left(  x\right)  \right)
\right\}  ,\qquad v\left(  r\right)  :=\int_{r}^{b-a}ds/\gamma^{d}\left(
s\right)  \label{pot}%
\end{equation}
(see Section 2 for the definition of capacity). In particular, from
(\ref{lowerc}) it follows that if $\gamma$ is strictly concave near zero and
$1/\gamma^{d}$ is integrable at $0$ (i.e. $v$ is bounded), then the process
$B$ hits points (singletons) with positive probability. In Theorem \ref{tcap},
no additional restrictions on $\gamma$ are needed, unlike \cite{Bierme:09} who
require commensurably to a power function.

In preparation for proving Hausdorff measure upper bounds for the hitting
probabilities in Section 4, we establish a result in \textbf{Section 3} for
hitting probabilities in dimension 1, which is interesting in its own right
(see Proposition \ref{hitd1} and Corollary \ref{hitd1cor}): under a mild
technical condition on $\gamma$ which is satisfied for all examples of
interest, the probability for a path of $B$ to hit a point between positive
times $a$ and $b$ is bounded above by a multiple of $\gamma\left(  b-a\right)
$. We use this as one ingredient in the proof of the following Hausdorff
measure upper bound for hitting probabilities (Theorem \ref{t3}), in
\textbf{Section 4}: if the function
\begin{equation}
\varphi(x):=s^{d}/\gamma^{-1}(x)
\end{equation}
is right-continuous and non-decreasing near $0$ with $\lim_{0+}\varphi=0$,
then for all $0<a<b<\infty$ and $M>0$, there exists a constant $C>0$ depending
only on $a,b$, the law of $B$, and $M$, such that for any Borel set
$A\subset\lbrack-M,M]^{d}$,
\begin{equation}
\P ({B}([a,b])\cap A\neq\varnothing)\leq C{\mathcal{H}}_{\varphi
}(A)\label{h11}%
\end{equation}
where ${\mathcal{H}}_{\varphi}$ is the Hausdoff measure based on $\varphi$
(see Section 4 for a definition), as opposed to the classical
power-scale-based Hausdorff ${\mathcal{H}}_{r}$ measure based on the function
$x\mapsto x^{r}$. As we mentioned before, a covering argument is needed to
prove this theorem, similar to what was done in \cite[Theorem 3.1]{Dalang:07}.
In addition to this and to the estimate of Section 3, a new ingredient we use
is the Malliavin-calculus-based estimation of the density of the random
variable $Z:=\inf_{s\in\lbrack a,b]}\left\vert B\left(  s\right)
-z\right\vert $ in one dimension, where $z$ is a fixed point in space. This
random variable has an atom at $0$, which we estimate in Section 3, and a
bounded density elsewhere, which we prove by adapting a quantitative density
formula first established in \cite[Theorem 3.1]{Nourdin:09} which uses the
Malliavin calculus. As a consequence, we show that the probability for $B$'s
path to reach a ball of radius $\varepsilon$ between times $a$ and $b$ is
bounded above by the bound from Section 3 plus a constant multiple of
$\varepsilon$, where the dependence of this constant on $\left\vert
z\right\vert $ is given explicitly.

In the case of fBm when $d>1/H$, it was previously known that the upper
Hausdorff measure bound (\ref{h11}) applies with $\varphi\left(  x\right)
=x^{d-1/H}$, and the capacity lower bound (\ref{lowerc}) uses the Newtonian
kernel $\mathrm{K}\left(  x\right)  =x^{1/H-d}=1/\varphi(x)$, i.e.
$1/\mathrm{K}=\varphi$. At the end of Section 4, we study this phenomenon
further in our general case, to give broad conditions, not related to power
scaling, under which the bounds (\ref{lowerc}) and (\ref{h11}) hold with
$1/\mathrm{K}=\varphi$.

\textbf{Section 5} is devoted to studying examples of applications of our
general theorems. We consider the class of processes satisfying (\ref{e1}) and
(\ref{e1v}) with $\gamma$ defined near $0$ by
\[
\gamma(r)=\gamma_{H,\beta}\left(  r\right)  :=r^{H}\log^{\beta}(\frac{1}{r}),
\]
for some $\beta\in\mathbb{R},H\in(0,1)$, or $H=1,\beta>0$, or $H=0,\beta
<-1/2$. In Theorem \ref{texa} and Corollary \ref{corex1} we prove that the
bounds (\ref{lowerc}) and (\ref{h11}) hold with $1/\mathrm{K}\left(  x\right)
=\varphi\left(  x\right)  =x^{d-\frac{1}{H}}\log^{\beta/H}(1/x)$ as soon as
$d>1/H$, or as soon as $d=1/H$ and $\beta<0$. However, if $d=1/H$ and
$\beta\in\lbrack0,1/d)$, the upper bound is not established, and the function
for the lower bound must be changed to $\varphi(x)=\log^{\beta/H-1}\left(
1/x\right)  $. If $d=1/H$ and $\beta\geq1/d$, or if $d<1/H$, the lower bound
holds with $\varphi\equiv1$. In this last case, this implies that $B$ hits
singletons with positive probability. We show more generally that $B$ hits
singletons with positive probability as soon as $1/\gamma^{d}$ is integrable
at $0$, whereas $B$ hits singletons with probability zero as soon as
$r^{-1}=o\left(  1/\gamma^{d}\left(  r\right)  \right)  $. We finish Section 5
with applications to estimating probabilities of hitting Cantor sets for
various combinations of parameters when $\gamma=\gamma_{H,\beta}$.

The applications in Section 5 are particularly revealing in the so-called
\textquotedblleft critical\textquotedblright\ case, where $H=1/d$. Unlike in
the H\"{o}lder scale, where this critical case is represented only by fBm or
similar processes, here we have an entire scale of processes as $\beta$ ranges
in all of $\mathbb{R}$, the fBm corresponding to $\beta=0$. Our results imply
that if $\beta<0$ (processes that are more regular than critical fBm, while
being indistinguishable from it in the H\"{o}lder scale) then a.s. $B$ does
not hit points, while if $\beta\geq1/d$ (processes that are less regular than
fBm, while also being indistinguishable from it in the H\"{o}lder scale), then
$B$ hits points with positive probability. The gap corresponding to the range
$\beta\in\lbrack0,1/d)$ indicates that there is presumably a slight
inefficiency in at least one of the two estimation methods for hitting
probability (Hausdorff measure and capacity). This thought is not new; what is
more important here is that our theorems give a precise quantification of the
inefficiency: it is of no more than a log order, which, in practice, could be
considered difficult to detect. Arguably, this provides support for continuing
to study both estimation methods in the future.

\section{Lower capacity bound for the hitting probabilities}

Recall that $B=(B(t),t\in\mathbb{R}_{+})$ is a centered continuous Gaussian
process in $\mathbb{R}$ with variance $\gamma^{2}(t)$ for some continuous
strictly increasing function $\gamma:\mathbb{R}_{+}\rightarrow\mathbb{R}_{+}$
with $\lim_{0}\gamma=0$, such that for some constant $\ell\geq1$ and for all
$s,t\geq0$,
\[
(1/\ell)\gamma^{2}(|t-s|)\leq\mathrm{E}[|B(t)-B(s)|^{2}]\leq\ell\gamma
^{2}(|t-s|).
\]

The aim of this section is to obtain a lower bound for the probability that a
$d$-dimensional vector of iid copies of $B$, also denoted by $B$, hits a set
$A\subset\mathbb{R}^{d}$ in terms of the capacity of $A$, a concept from
potential theory. In this context, it is customary to say that any measurable
function $\mathrm{F}:\mathbb{R}^{d}\times\mathbb{R}^{d}\rightarrow(0,\infty]$
can serve as a so-called \emph{potential kernel}. In this article, we will
focus on the case where $\mathrm{F}(x,y)=\mathrm{K}(|x-y|)$, and $\mathrm{K}$
is a positive, non-increasing, continuous function in $\mathbb{R}_{+}%
\setminus0$ with $\lim_{0}\mathrm{K}\leq+\infty$.

The capacity of a Borel set $A\subset\mathbb{R}^{d}$ with respect to a
potential kernel $\mathrm{K}$ is defined as
\[
\mathcal{C}_{\mathrm{K}}(A):=\left[  \inf_{\mu\in\mathcal{P}(A)}{\mathcal{E}%
}_{\mathrm{K}}(\mu)\right]  ^{-1},
\]
where $\mathcal{P}(A)$ denotes the set of probability measures with support in
$A$, and ${\mathcal{E}}_{\mathrm{K}}(\mu)$ denotes the energy of a measure
$\mu\in\mathcal{P}(A)$ with respect to the kernel $\mathrm{K}$, which is
defined as
\[
{\mathcal{E}}_{\mathrm{K}}(\mu):=\iint_{\mathbb{R}^{d}\times\mathbb{R}^{d}%
}\mathrm{K}(|x-y|)\,\mu(dx)\,\mu(dy).
\]
By convention for computing $\mathcal{C}_{\mathrm{K}}(A)$, we let
$1/\infty:=0$. The Newtonian $\beta$-kernel is the potential kernel defined
as
\[
\label{k}\mathrm{K}_{\beta}(r):=%
\begin{cases}
r^{-\beta} & \text{if $\beta>0$},\\
\log\left(  \frac{e}{r\wedge1}\right)  & \text{if $\beta=0$},\\
1 & \text{if $\beta<0$}.
\end{cases}
\]

The $\beta$-capacity of a Borel set $A\subset\mathbb{R}^{d}$, denoted
$\mathcal{C}_{\beta}(A)$, and the $\beta$-energy of a measure $\mu
\in\mathcal{P}(A)$, denoted ${\mathcal{E}}_{\beta}(\mu)$, are the capacity and
the energy with respect to the Newtonian $\beta$-kernel $\mathrm{K}_{\beta}$.

Let us consider the following additional hypotheses on $\gamma$.

\begin{hyp}
\label{h0} Recall the constant $\ell$ in (\ref{e1}). The increasing function
$\gamma$ in (\ref{e1}) is concave in a neighborhood of the origin, and for all
$0<a<b<\infty$, there exists $\varepsilon>0$ such that $\gamma^{\prime}\left(
\varepsilon+\right)  >\sqrt{\ell}~\gamma^{\prime}\left(  a-\right)  $.
\end{hyp}

\begin{hyp}
\label{h1} Recall the constant $\ell$ in (\ref{e1}). For all $0<a<b<\infty$,
there exists $\varepsilon>0$ and $c_{0}\in(0,1/\sqrt{\ell})$, such that for
all $s,t\in\lbrack a,b]$ with $0<t-s\leq\varepsilon$,
\[
\gamma(t)-\gamma(s)\leq c_{0}\gamma(t-s).
\]

\end{hyp}

Since a concave function has a derivative almost everywhere with finite left
and right limits for this derivative everywhere, the strict inequality in
Hypothesis \ref{h0} is simply saying that $\gamma$ is strictly concave near
the origin. In all the examples that have been mentioned up to now, and the
ones which we consider later in this article (see Section 5, in particular),
we have $\gamma^{\prime}\left(  0+\right)  =+\infty$, from which Hypothesis
\ref{h0} reduces to simply requiring concavity near the origin. That concavity
is satisfied in all examples we look at.

To construct a Gaussian process which fails to satisfy Hypothesis \ref{h0},
one has to resort to processes which are somewhat pathological, and whose
smoothness are such that hitting probabilities become trivial. For instance,
the fBm $B^{1}$ with parameter $H=1$ satisfies $\mathrm{E}\left[  \left(
B^{1}\left(  t\right)  -B^{1}\left(  s\right)  \right)  ^{2}\right]
=\left\vert t-s\right\vert ^{2}$ which implies that there exists a standard
normal rv $Z$ such that $B^{1}\left(  t\right)  =tZ$, and we have
$\gamma\left(  r\right)  =r$ in (\ref{e1}).

Hypothesis \ref{h1} implies a lower bound for the conditional variance of the
random variable $B(t)$ given $B(s)$, provided in Lemma \ref{lem1} below. This
property of the conditional variance can be referred to as a two-point local
nondeterminism (see \cite{Bierme:09}). We first prove that Hypothesis \ref{h0}
implies Hypothesis \ref{h1}, and that under the stronger assumption (satisfied
in all our examples) that $\gamma^{\prime}\left(  0+\right)  =+\infty$, the
constant $c_{0}$ in Hypothesis \ref{h1} can be chosen arbitrarily small.

\begin{lemma}
\label{lem0}Hypothesis \ref{h0} implies Hypothesis \ref{h1}. If moreover
$\gamma^{\prime}\left(  0+\right)  =+\infty$, then for all $0<a<b<\infty$, and
all $c_{0}>0$, there exists $\varepsilon>0$ such that for all $s,t\in\lbrack
a,b]$ with $0<t-s\leq\varepsilon$, $\gamma(t)-\gamma(s)\leq c_{0}\gamma(t-s).$
\end{lemma}

\begin{proof}
By concavity of $\gamma$, and the fact that $\gamma\left(  0\right)  =0$, for
any $s,t\geq a$ such that $0<t-s\leq\varepsilon$
\[
\frac{\gamma\left(  t\right)  -\gamma\left(  s\right)  }{t-s}\leq
\gamma^{\prime}\left(  a-\right)  =\gamma^{\prime}\left(  \varepsilon+\right)
\frac{\gamma^{\prime}\left(  a-\right)  }{\gamma^{\prime}\left(
\varepsilon+\right)  }\leq\frac{\gamma\left(  t-s\right)  }{t-s}\cdot
\frac{\gamma^{\prime}\left(  a-\right)  }{\gamma^{\prime}\left(
\varepsilon+\right)  }%
\]
which proves the conclusion of Hypothesis \ref{h1} by taking $c_{0}%
:=\gamma^{\prime}\left(  a-\right)  /\gamma^{\prime}\left(  \varepsilon
+\right)  $ which is strictly less than $1/\sqrt{\ell}$ by Hypothesis
\ref{h0}. The second statement of the lemma follows using the same proof by
noting that for any fixed $a>0$, one can make $\gamma^{\prime}\left(
a-\right)  /\gamma^{\prime}\left(  \varepsilon+\right)  $ as small as desired
by choosing $\varepsilon$ small enough.
\end{proof}

\begin{lemma}
\label{lem1} Assume Hypothesis \ref{h1}. Then for all $0<a<b<\infty$, there
exists $\varepsilon>0$ and a positive constant $c(a,b)$ depending only on
$a,b$ and the law of the scalar process $B$, such that for all $s,t\in\lbrack
a,b]$ with $|t-s|\leq\varepsilon$,
\begin{equation}
\mathrm{Var}(B(t)|B(s))\geq c\left(  a,b\right)  \gamma^{2}(|t-s|). \label{20}%
\end{equation}

More specifically, assume simply that $\gamma$ is concave in a neighborhood of
the origin and $\gamma^{\prime}\left(  0+\right)  =+\infty$; then the above
conclusion holds for any $c\left(  a,b\right)  <\gamma^{4}\left(  a\right)
/\left(  2\ell~\gamma^{4}\left(  b\right)  \right)  $, for some $\varepsilon
>0$ small enough.
\end{lemma}

\begin{proof}
Recall that
\[
\mathrm{Var}(B(t)|B(s))=\gamma^{2}(t)(1-\rho^{2}(s,t)),
\]
where $\rho(s,t)$ denotes the correlation coefficient between $B(s)$ and
$B(t)$, that is,
\[
\rho(s,t):=\frac{\sigma(s,t)}{\gamma(s)\gamma(t)},
\]
and $\sigma(s,t)$ denotes the covariance of $B(s)$ and $B(t)$, that is,
\[
\sigma(s,t):=\mathrm{E}[B(s)B(t)].
\]
None of these functions depend on $i$.

Hence, as $\gamma$ is increasing, $\gamma^{2}\left(  t\right)  \geq\gamma
^{2}\left(  a\right)  $ and it suffices to find a lower bound for $1-\rho
^{2}(s,t)$. Using the lower bound in (\ref{e1}), we have that
\[%
\begin{split}
1-\rho(s,t)  &  =\frac{\gamma(s)\gamma(t)-\frac{1}{2}(\gamma^{2}(t)+\gamma
^{2}(s)-\delta^{2}(s,t))}{\gamma(s)\gamma(t)}\\
&  =\frac{\delta^{2}(s,t)-(\gamma(t)-\gamma(s))^{2}}{2\gamma(s)\gamma(t)}\\
&  \geq\frac{(1/\ell)\gamma^{2}(|t-s|)-(\gamma(t)-\gamma(s))^{2}}%
{2\gamma(s)\gamma(t)},
\end{split}
\]
where $\delta^{2}(s,t):=\mathrm{E}[(B(t)-B(s))^{2}]$.

Next, appealing to Hypothesis \ref{h1} and the fact that $\gamma$ is
increasing, we get that for all $s,t\in\lbrack a,b]$ with $|t-s|\leq
\varepsilon$,
\begin{align*}
1-\rho(s,t)  &  \geq\frac{\frac{1}{\ell}-c_{0}^{2}}{2\gamma\left(  s\right)
\gamma\left(  t\right)  }\gamma^{2}(|t-s|)\\
&  \geq\frac{\frac{1}{\ell}-c_{0}^{2}}{2\gamma^{2}\left(  b\right)  }%
\gamma^{2}(|t-s|).
\end{align*}
On the other hand, by (\ref{e1}) and the fact that $\gamma$ is increasing, for
all $s,t\in\lbrack a,b]$ with $|t-s|\leq\varepsilon$,
\begin{align*}
1+\rho(s,t)  &  =\frac{(\gamma(t)+\gamma(s))^{2}-\delta^{2}(s,t)}%
{2\gamma(s)\gamma(t)}\\
&  \geq\frac{2\gamma^{2}\left(  a\right)  }{\gamma^{2}\left(  b\right)
}-\frac{\ell}{2\gamma^{2}\left(  a\right)  }\gamma^{2}\left(  \varepsilon
\right)  .
\end{align*}
Since $\lim_{0}\gamma=0$, we can choose $\varepsilon$ sufficiently small such
that the last displayed line above is bounded below by $\gamma^{2}\left(
a\right)  /\gamma^{2}\left(  b\right)  $. Therefore, for such $\varepsilon$,
and all $\left\vert t-s\right\vert <\varepsilon$, $s,t\in\lbrack a,b]$, we
have%
\[
1-\rho^{2}(s,t)=(1+\rho(s,t))(1-\rho(s,t))\geq\frac{\gamma^{2}\left(
a\right)  }{\gamma^{2}\left(  b\right)  }\frac{\frac{1}{\ell}-c_{0}^{2}%
}{2\gamma^{2}\left(  b\right)  }\gamma^{2}(|t-s|),
\]
which concludes the proof of the lemma's first statement. The lemma's second
statement also follows because, by Lemma \ref{lem0} and the discussion
following the introduction of Hypotheses \ref{h0} and \ref{h1}, we can choose
$c_{0}$ above arbitrarily small.
\end{proof}

We next obtain a lower bound for the hitting probabilities of $B$ in terms of
capacity. Our proof employs a strategy developed in \cite[Theorem
2.1]{Bierme:09} where the authors obtain a lower bound in terms of capacity,
for the hitting probabilities of a general class of multi-parameter
anisotropic Gaussian random fields within a power scale (see also
\cite{Testard:86}). Unlike in \cite[Theorem 2.1]{Bierme:09} our result is not
formulated using -- and our proofs are not based on -- proximity in law to a
process with self-similar increments (the power scale), or on related concepts
such as H\"{o}lder continuity or function index. The capacity kernel
identified in the next result is intrinsic to the law of the Gaussian process
$B$, since it is computed using only the function $\gamma$ in assumptions
(\ref{e1}) and (\ref{e1v}), not an exogenously identified H\"{o}lder exponent
or index.

\begin{thm}
\label{tcap}Let $\mathrm{K}(x):=\max\left\{  1;v\left(  \gamma^{-1}\left(
x\right)  \right)  \right\}  $ where $v\left(  r\right)  :=\int_{r}%
^{b-a}ds/\gamma^{d}\left(  s\right)  $. Assume Hypothesis \ref{h0} holds. Then
for all $0<a<b<\infty$ and $M>0$, there exists a constant $C>0$ depending only
on $a,b,M$ and the law of $B$, such that for any Borel set $A\subset
\lbrack-M,M]^{d}$
\[
C\mathcal{C}_{\mathrm{K}}(A)\leq\P (B([a,b])\cap A\neq\varnothing).
\]

\end{thm}

\begin{rmk}
Recall that Hypothesis \ref{h0} holds as soon as $\gamma$ is concave near $0$
with $\gamma^{\prime}\left(  0+\right)  =+\infty$.
\end{rmk}

\begin{rmk}
\label{hitprem}When $1/\gamma^{d}$ is integrable at $0$, $\mathrm{K}$ is
bounded by $\mathrm{K}_{\infty}:=\max\left(  1;\int_{0}^{b-a}1/\gamma
^{d}\right)  $, and therefore, replacing $C$ by $C/\mathrm{K}_{\infty}$, we
may replace $\mathcal{C}_{\mathrm{K}}(A)$ by $\mathcal{C}_{1}(A)$ in Theorem
\ref{tcap}. As $\mathcal{C}_{1}(A)=1$ for every non-empty set $A$, Theorem
\ref{tcap} shows that if $1/\gamma^{d}$ is integrable at $0$, then the process
$B$ hits points with positive probability.
\end{rmk}

\begin{rmk}
\label{GermanRem}It is useful to compare our condition $\int_{0}\gamma
^{-d}\left(  r\right)  dr<\infty$ in Remark \ref{hitprem} for hitting points
to a classical strategy for establishing such non-polarity of points, whose
elements are in German and Horowitz \cite{German:80}. If condition
\cite[(21.10)]{German:80} holds, by \cite[Theorem 21.9]{German:80}, the
process $B$ has a square integrable local time. Moreover if $B$ has the local
non-determinism property and condition \cite[(25.13)]{German:80} holds, then
\cite[Proposition (25.12)(c) and Theorem (26.1)]{German:80} imply that the
local time of $B$ is jointly continuous. Then a fairly classical argument can
be used to prove that the two conditions \cite[(21.10) and (25.13)]{German:80}
and would imply that the process $B$ hits points in $\mathbb{R}^{d}$ with
positive probability.

Our assumption $\int_{0}\gamma^{-d}\left(  r\right)  dr<\infty$ implies
condition \cite[(21.10)]{German:80}. Thus, to compare the assumptions in the
classical local-time-based strategy to our assumption for hitting points, it
is sufficient to study the relationship bewteen condition \cite[(25.13)]%
{German:80} and our condition $\int_{0}\gamma^{-d}\left(  r\right)  dr<\infty
$. Let $\Gamma(t,s)$ denote the covariance matrix of $B_{t}-B_{s}$ and
$\Delta(t,s)$ its determinant; since the components of $B$ are independent,
$\Delta\left(  t,s\right)  $ is commensurate with $\gamma\left(  t-s\right)
^{2d}$. Condition \cite[(25.13)]{German:80} requires that there exist
$\varepsilon\in(0,1]$ such that
\[
\sup_{s\in\lbrack0,1]}\int_{0}^{1}\frac{dt}{(\Delta(t,s))^{\frac{1}%
{2}+\epsilon}}<\infty
\]
The integral above is bounded below by $\ell^{-1}\int_{0}^{1}\gamma\left(
r\right)  ^{-d-\varepsilon/2}dr$. We conclude that condition \cite[(25.13)]%
{German:80} implies a stronger assumption than our condition $\int_{0}%
\gamma^{-d}\left(  r\right)  dr<\infty$.
\end{rmk}

\begin{proof}
[Proof of Theorem \ref{tcap}]\emph{Step 1: Setup and general strategy.}

Note that $\mathrm{K}$, as defined in the statement of the theorem, is a
bonafide univariate potential kernel since it is non-increasing and continuous
on $\mathbb{R}_{+}\setminus0$ with $\lim_{0}\mathrm{K=K}_{\infty}=\max\left(
1;\int_{0}^{b-a}\gamma^{-d}\right)  \leq+\infty$. Assume that $\mathcal{C}%
_{\mathrm{K}}(A)>0$ otherwise there is nothing to prove. This implies the
existence of a probability measure $\mu\in\mathcal{P}(A)$ such that
\begin{equation}
\mathcal{E}_{\mathrm{K}}(\mu)\leq\frac{2}{\mathcal{C}_{\mathrm{K}}(A)}.
\label{cap2}%
\end{equation}

Consider the sequence of random measures $(\nu_{n})_{n \geq1}$ on $[a,b]$
defined as
\[
\nu_{n}(dt)=\int_{\mathbb{R}^{d}} (2 \pi n)^{d/2} \exp\left(  -\frac{n \vert
B(t)-x \vert^{2}}{2}\right)  \mu(dx) dt.
\]
By the Fourier inversion theorem
\[
\nu_{n}(dt)=\int_{\mathbb{R}^{d}} \int_{\mathbb{R}^{d}} \exp\left(
-\frac{\vert\xi\vert^{2}}{2n}+i \langle\xi, B(t)-x \rangle\right)  d\xi\mu(dx)
dt.
\]
Denote the total mass of $\nu_{n}$ by $\vert\nu_{n} \vert:=\nu_{n}([a,b])$. We
claim that
\begin{equation}
\label{claim}\mathrm{E}(\vert\nu_{n} \vert) \geq c_{1}, \qquad\text{and}
\qquad\mathrm{E}(\vert\nu_{n} \vert^{2}) \leq c_{2} {\mathcal{E}}_{\mathrm{K}%
}(\mu),
\end{equation}
where the constants $c_{1}$ and $c_{2}$ are independent of $n$ and $\mu$.

We first have
\[%
\begin{split}
\mathrm{E}(\vert\nu_{n} \vert)  &  =\int_{a}^{b} \int_{\mathbb{R}^{d}}
\int_{\mathbb{R}^{d}} \exp\left(  -\frac{\vert\xi\vert^{2}}{2}\left(  \frac
{1}{n}+\gamma^{2}(t)\right)  -i \langle\xi, x \rangle\right)  d\xi\mu(dx) dt\\
&  \geq\int_{a}^{b} \int_{\mathbb{R}^{d}} \frac{(2 \pi)^{d/2}}{(1+\gamma
^{2}(t))^{d/2}} \exp\left(  -\frac{\vert x \vert^{2}}{2 \gamma^{2}(t)}\right)
\mu(dx) dt\\
&  \geq\int_{a}^{b} \frac{(2 \pi)^{d/2}}{(1+\gamma^{2}(b)^{d/2}} \exp\left(
-\frac{dM^{2}}{2 \gamma^{2}(a)}\right)  dt=:c_{1},
\end{split}
\]
where the second inequality follows because $A \subset[-M, M]^{d}$. This
proves the first inequality in (\ref{claim}).

We next prove the second inequality in (\ref{claim}). We have that
\[%
\begin{split}
&  \mathrm{E}(|\nu_{n}|^{2})=\int_{a}^{b}\int_{a}^{b}\int_{\mathbb{R}^{2d}%
}\int_{\mathbb{R}^{2d}}e^{-i(\langle\xi,x\rangle+\langle\eta,y\rangle)}\\
&  \qquad\qquad\qquad\qquad\times\exp\left(  -\frac{1}{2}(\xi,\eta)\Gamma
_{n}(s,t)(\xi,\eta)^{T}\right)  d\xi d\eta\mu(dx)\mu(dy)dsdt,
\end{split}
\]
where $\Gamma_{n}(s,t)$ denotes the $2d\times2d$ matrix
\[
\Gamma_{n}(s,t)=n^{-1}I_{2d}+\text{Cov}(B(s),B(t)),
\]
$I_{2d}$ denotes the $2d\times2d$ identity matrix, and $\text{Cov}
(B(s),B(t))$ denotes the $2d\times2d$ covariance matrix of $\left(
B(s),B(t)\right)  $.

Observe that $B(t)=(B_{1}(t),...,B_{d}(t))$ where the $B_{i}(t)$'s are the $d$
independent coordinate processes of $B(t)$. Hence,
\begin{equation}
\label{equa6}%
\begin{split}
&  \mathrm{E}(\vert\nu_{n} \vert^{2}) =\int_{a}^{b} \int_{a}^{b}
\int_{\mathbb{R}^{2d}} \int_{\mathbb{R}^{2d}} e^{-i \sum_{j=1}^{d} (\xi_{j}
x_{j} +\eta_{j} y_{j})}\\
&  \times\exp\left(  - \frac12 \sum_{j=1}^{d} ((\gamma^{2}(s)+\frac{1}%
{n})x_{j}^{2}+2\sigma(s,t) x_{j} y_{j}+(\gamma^{2}(t)+\frac{1}{n})y_{j}^{2})
\right)  d\xi d\eta\mu(dx) \mu(dy) ds dt\\
&  =\int_{a}^{b} \int_{a}^{b} \int_{\mathbb{R}^{2d}} \prod_{j=1}^{d}
\bigg(\int_{\mathbb{R}^{2}} e^{-i (\xi_{j} x_{j} +\eta_{j} y_{j})}\\
&  \times\exp\left(  - \frac12 ((\gamma^{2}(s)+\frac{1}{n})x_{j}^{2}%
+2\sigma(s,t) x_{j} y_{j}+(\gamma^{2}(t)+\frac{1}{n})y_{j}^{2}) \right)  d\xi
d\eta\bigg) \mu(dx) \mu(dy) ds dt,
\end{split}
\end{equation}
where recall that $\sigma(s,t)$ denotes the covariance of $B_{j}(s)$ and
$B_{j}(t)$ which does not depend on $j$.

Next observe that integral inside the product in (\ref{equa6}) is equal to
\[%
\begin{split}
\int_{\mathbb{R}^{2}} e^{-i (\xi_{j} x_{j}+\eta_{j} y_{j})} \exp\left(  -
\frac12 (\xi_{j}, \eta_{j}) \Phi_{n}(s,t) (\xi_{j}, \eta_{j})^{T} \right)
d\xi d\eta
\end{split}
\]
where $\Phi_{n}(s,t)$ denotes the $2 \times2$ matrix
\[
\Phi_{n}(s,t)=n^{-1} I_{2}+\text{Cov}(B_{j}(s), B_{j}(t)),
\]
which is independent of $j$.

Since $\Phi_{n}(s,t)$ is positive definite, we have
\begin{equation}%
\begin{split}
&  \int_{\mathbb{R}^{2}}e^{-i(\xi_{j} x_{j}+\eta_{j} y_{j} )}\exp\left(
-\frac{1}{2}(\xi_{j},\eta_{j})\Phi_{n}(s,t)(\xi_{j} ,\eta_{j})^{T}\right)
d\xi d\eta\\
&  \qquad=\frac{2\pi}{\sqrt{\text{det}(\Phi_{n}(s,t))}}\exp\left(  -\frac
{1}{2}(x_{j},y_{j}\Phi_{n}^{-1}(s,t)(x_{j},y_{j})^{T}\right)  .
\end{split}
\label{equa5}%
\end{equation}
\vspace*{0.1in}

\emph{Step 2: using the method of \cite{Bierme:09} near the diagonal}. We now
follow the proof of \cite[Lemma 11]{Bierme:09}, in order to show that for all
$s,t\in\lbrack a,b]$ with $|t-s|\leq\varepsilon$
\begin{equation}
(x_{j},y_{j})\Phi_{n}^{-1}(s,t)(x_{j},y_{j})^{T}\geq c_{3}\frac{(x_{j}%
-y_{j})^{2}}{\text{det}(\Phi_{n}(s,t))}, \label{equa}%
\end{equation}
for some constant $c_{3}>0$ depending only on $(a,b)$, and $\varepsilon>0$ as
in Lemma \ref{lem1}.

First remark that
\[
(x_{j},y_{j}) \Phi^{-1}_{n}(s,t) (x_{j},y_{j})^{T} \geq\frac{1}{\text{det}%
(\Phi_{n}(s,t))} \mathrm{E} \left(  (x_{j} B_{j}(t)-y_{j} B_{j}(s))^{2}
\right)  .
\]
Thus, in order to show (\ref{equa}), it suffices to prove that
\begin{equation}
\label{equa2}\mathrm{E} \left(  (x_{j} B_{j}(t)-y_{j} B_{j}(s))^{2} \right)
\geq c_{3} (x_{j}-y_{j})^{2}.
\end{equation}
Next observe that in order to prove (\ref{equa2}) it suffices to show that
\begin{equation}
\label{equa3}\mathrm{E}((B_{j}(t))^{2}) \geq c_{4},
\end{equation}
and that
\begin{equation}
\label{equa4}\frac{\text{det}(\text{Cov}(B_{j}(s), B_{j}(t)))}{\mathrm{E}%
((B_{j}(s)-B_{j}(t))^{2})} \geq c_{5},
\end{equation}
for some constants $c_{4}, c_{5}>0$ only depending on $(a,b)$.

Observe that (\ref{equa3}) holds by (\ref{e1v}) for all $t\in\lbrack a,b]$ as
$\gamma^{2}$ is increasing. On the other hand, by (\ref{e1}) the denominator
in (\ref{equa4}) is bounded by $\ell\gamma^{2}(|t-s|)$. Moreover, by
Hypothesis \ref{h0} and Lemmas \ref{lem0} and \ref{lem1}, the numerator in
(\ref{equa4}) satisfies that for all $s,t\in\lbrack a,b]$ with $|t-s|\leq
\varepsilon$
\[
\text{det}(\text{Cov}(B_{j}(s),B_{j}(t)))=\gamma^{2}(s)\text{Var}%
(B_{j}(t)|B_{j}(s))\geq c_{6}\gamma^{2}(|t-s|),
\]
for some $c_{6}>0$ depending only on $(a,b)$. Therefore, (\ref{equa4}) holds
for all $s,t\in\lbrack a,b]$ with $|t-s|\leq\varepsilon$. Thus, (\ref{equa})
holds true.

Now, replacing (\ref{equa}) into (\ref{equa5}), returning to the computation
in (\ref{equa6}), and appealing to Fubini's theorem, we get that
\[
\mathrm{E}(|\nu_{n}|^{2})\leq I_{1}+I_{2},
\]
where
\[%
\begin{split}
I_{1}  &  :=\int_{\mathbb{R}^{2d}}\iint_{D(\varepsilon)}\frac{(2\pi)^{d}%
}{(\sqrt{\text{det}(\Phi_{n}(s,t))})^{d}}\exp\left(  -\frac{c_{3}}{2}%
\frac{|x-y|^{2}}{\text{det}(\Phi_{n}(s,t))}\right)  dsdt\mu(dx)\mu(dy),\\
I_{2}  &  :=\int_{\mathbb{R}^{2d}}\iint_{[a,b]^{2}\setminus D(\varepsilon
)}\frac{(2\pi)^{d}}{(\sqrt{\text{det}(\Phi_{n}(s,t))})^{d}}dsdt\mu(dx)\mu(dy),
\end{split}
\]
and $D(\varepsilon)=\{(s,t)\in\lbrack a,b]^{2}:|t-s|\leq\varepsilon\}$%
.\vspace*{0.1in}

\emph{Step 3: Bounding the off-diagonal contribution to }$\mathrm{E}(|\nu
_{n}|^{2})$.

We start by bounding $I_{2}$. Observe that%
\begin{align*}
\det(\Phi_{n}(s,t))  &  =\gamma^{2}(s)\gamma^{2}(t)-\sigma^{2}(s,t)+\frac
{\gamma^{2}(s)}{n}+\frac{\gamma^{2}(t)}{n}+\frac{1}{n^{2}}\\
&  \geq\gamma^{2}(s)\gamma^{2}(t)-\sigma^{2}(s,t).
\end{align*}
By the Cauchy-Schwartz inequality, the function $\left(  s,t\right)
\mapsto\gamma^{2}(s)\gamma^{2}(t)-\sigma^{2}(s,t)$ is non-negative, and since
$\gamma\left(  r\right)  =0\iff r=0$, this function is strictly positive and
continuous away from the diagonal $\left\{  s=t\right\}  $. Therefore, for all
$s,t\in\lbrack a,b]$ with $|t-s|>\varepsilon$,
\[
\text{det}(\Phi_{n}(s,t))\geq c_{7},
\]
for some constant $0<c_{7}(a,b)$. Hence, we get that
\[
\iint_{\lbrack a,b]^{2}\setminus D(\varepsilon)}\frac{(2\pi)^{d}}%
{(\sqrt{\text{det}(\Phi_{n}(s,t))})^{d}}dsdt\leq c_{8},
\]
where the constant $c_{8}$ only depends on $(a,b)$. Therefore from the
definition of $I_{2}$ in step 2, and since $K$ is non-increasing, we get%
\begin{align*}
I_{2}  &  \leq c_{8}\int_{\mathbb{R}^{2d}}\mu(dx)\mu(dy)=c_{8}\int
_{\mathbb{R}^{2d}}\mu(dx)\mu(dy)\frac{K\left(  \left\vert x-y\right\vert
\right)  }{K\left(  \left\vert x-y\right\vert \right)  }\\
&  \leq\frac{c_{8}}{K(2M)}\mathcal{E}_{\mathrm{K}}(\mu)\leq c_{8}%
\mathcal{E}_{\mathrm{K}}(\mu).
\end{align*}
This prove the part of claim (\ref{claim}) corresponding to $I_{2}$, i.e.
$I_{2}\leq const\cdot\mathcal{E}_{\mathrm{K}}\left(  \mu\right)  $%
.\vspace*{0.1in}

\emph{Step 4: Bounding the diagonal contribution to }$\mathrm{E}(|\nu_{n}%
|^{2})$.

We next treat $I_{1}$. Observe that if $\text{det}(\Phi_{n}(s,t))<|x-y|^{2}$,
using the inequality $x^{d/2}e^{-cx}\leq c^{\prime}$ valid for all $x>0$, it
yields that
\[
\frac{(2\pi)^{d}}{(\sqrt{\text{det}(\Phi_{n}(s,t))})^{d}}\exp\left(
-\frac{c_{3}}{2}\frac{|x-y|^{2}}{\text{det}(\Phi_{n}(s,t))}\right)  \leq
\frac{c_{9}}{|x-y|^{d}}.
\]

Therefore, by Lemma \ref{lem1} we conclude that
\[%
\begin{split}
&  \iint_{D(\varepsilon)}\frac{(2\pi)^{d}}{(\sqrt{\text{det}(\Phi_{n}%
(s,t))})^{d}}\exp\left(  -\frac{c_{3}}{2}\frac{|x-y|^{2}}{\text{det}(\Phi
_{n}(s,t))}\right)  dsdt\\
&  \qquad\qquad\qquad\leq c_{10}\int_{a}^{b}\int_{a}^{b}\frac{1}{\max
(\gamma^{d}(|t-s|),|x-y|^{d})}dsdt.
\end{split}
\]

We next break the last integral into the regions $\{(s,t)\in\lbrack
a,b]^{2}:\gamma(|t-s|)\leq|x-y|)\}$ and $\{(s,t)\in\lbrack a,b]^{2}%
:\gamma(|t-s|)>|x-y|)\}$ and denote them by $J_{1}$ and $J_{2}$, respectively.
We first have that
\begin{equation}
J_{1} =\int_{a}^{b}\int_{\{t\in\lbrack a,b]:\gamma(|t-s|)\leq|x-y|)\}}\frac
{1}{|x-y|^{d}}dtds \leq c_{11}|x-y|^{-d}\gamma^{-1}(|x-y|), \label{J1ub}%
\end{equation}
where in the last inequality we have used the fact that $\gamma^{-1}$ is
strictly increasing.

On the other hand, using the change of variable $u=t-s$, we have that
\begin{equation}%
\begin{split}
J_{2}  &  =\int_{a}^{b}\int_{\{t\in\lbrack a,b]:\gamma(|t-s|)>|x-y|)\}}%
\frac{1}{\gamma^{d}(|t-s|)}dtds\\
&  \leq c_{12}\int_{\{u\in\lbrack0,b-a]:\gamma(u)>|x-y|)\}}\frac{1}{\gamma
^{d}(u)}du\\
&  =c_{12}\cdot v\left(  \gamma^{-1}\left(  \left\vert x-y\right\vert \right)
\right)  .
\end{split}
\end{equation}
where we again used the fact that $\gamma^{-1}$ is strictly increasing.
Integrating the last expression above against $\mu\left(  dx\right)
\mu\left(  dy\right)  $ over $D\left(  \varepsilon\right)  $ yields the upper
bound $c_{12}\cdot\mathcal{E}_{\mathrm{K}}\left(  \mu\right)  $, which is what
is required for the corresponding portion of the proof that $I_{1}\leq
const\cdot\mathcal{E}_{\mathrm{K}}\left(  \mu\right)  $.

To finish the proof that $I_{1}\leq const\cdot\mathcal{E}_{\mathrm{K}}\left(
\mu\right)  $, it is sufficient to show that
\[
J_{1}/v\left(  \gamma^{-1}\left(  \left\vert x-y\right\vert \right)  \right)
\]
is bounded. To prove this, given the upper bound in (\ref{J1ub}), and
reverting to $r=\gamma^{-1}\left(  \vert x-y \vert\right)  $ as our variable,
it is sufficient to prove that the function defined on $(0,b-a]$ by
\[
f\left(  r\right)  :=\frac{r/\gamma^{d}\left(  r\right)  }{v\left(  r\right)
}.
\]
is bounded. Since the functions in this ratio are all continuous, it is
sufficient to prove that $\lim_{0+}f<+\infty$. For this, let us consider two
different cases depending on whether $v$ is bounded or not. First, when $v$ is
bounded, as $1/\gamma^{d}$ is integrable at $0$ and $1/r$ is not, we have that
$\lim_{0+} \frac{1/\gamma^{d}(r)}{1/r}< +\infty$, and hence $\lim_{0+}f<
+\infty$. We now assume that $v$ is unbounded. We may then invoke the
following elementary one-sided extension of l'H\^{o}pital's rule for functions
$g,h$ whose derivatives have left and right limits everywhere: if $\lim
_{0+}h=+\infty$ and $\lim\sup_{0+}\max\left(  g^{\prime}\left(  r-\right)
;g^{\prime}\left(  r+\right)  \right)  /\min\left(  h^{\prime}\left(
r+\right)  ;h^{\prime}\left(  r-\right)  \right)  <+\infty$, then $\lim
\sup_{0+}g\left(  r\right)  /h\left(  r\right)  <+\infty$. With the
corresponding $g$ and $h$ from the definition of $f$ above, we find%
\begin{align*}
&  \frac{\max\left(  g^{\prime}\left(  r-\right)  ;g^{\prime}\left(
r+\right)  \right)  }{\min\left(  h^{\prime}\left(  r+\right)  ;h^{\prime
}\left(  r-\right)  \right)  }\\
&  =\frac{g^{\prime}\left(  r-\right)  }{h^{\prime}\left(  r\right)  }%
=\frac{-d\gamma^{-d-1}\left(  r\right)  \gamma^{\prime}\left(  r-\right)
r+1/\gamma^{d}\left(  r\right)  }{-1/\gamma^{d}\left(  r\right)  }\\
&  =-1+d\frac{r\gamma^{\prime}\left(  r-\right)  }{\gamma\left(  r\right)  }.
\end{align*}
By our concavity assumption in Hypothesis \ref{h0}, since $\gamma\left(
0\right)  =0$, we get that for every $r$ close enough to $0$, $\gamma\left(
r\right)  /r\geq\gamma^{\prime}\left(  r-\right)  $. This proves that the last
expression above is bounded above by $-1+d$. This finishes the proof that
$\lim_{0+}f<+\infty$, and therefore that $I_{1}\leq const\cdot\mathcal{E}%
_{\mathrm{K}}\left(  \mu\right)  $.\vspace*{0.1in}

\emph{Step 5: Conclusion.}

By Step 2, and the final estimates from Steps 3 and 4, the claim (\ref{claim})
is justified. Using (\ref{claim}) and the Paley-Zygmund inequality (cf.
\cite[p.8]{Kahane:85}), one can check that there exists a subsequence of the
sequence $(\nu_{n})_{n\geq1}$ that converges weakly to a finite measure $\nu$
which is positive with positive probability, satisfies (\ref{claim}), and is
supported on $[a,b]\cap(B)^{-1}(A)$. Therefore, using again the Paley-Zygmund
inequality, we conclude that
\[
\P (B([a,b])\cap A\neq\varnothing)\geq\P (|\nu|>0)\geq\frac{\mathrm{E}%
(|\nu|)^{2}}{\mathrm{E}(|\nu|^{2})}\geq\frac{c_{1}^{2}}{c_{2}\mathcal{E}%
_{\mathrm{K}}(\mu)}.
\]
Finally, (\ref{cap2}) finishes the proof of the Theorem.
\end{proof}

\section{The probability for a scalar Gaussian process to hit
points\label{hitd1sec}}

Recall that $B=(B(t),t\in\mathbb{R}_{+})$ is a centered continuous Gaussian
process in $\mathbb{R}$ with variance $\gamma^{2}(t)$ for some continuous
strictly increasing function $\gamma:\mathbb{R}_{+}\rightarrow\mathbb{R}_{+}$
with $\lim_{0}\gamma=0$, such that for some constant $\ell\geq1$ and for all
$s,t\geq0$,
\[
(1/\ell)\gamma^{2}(|t-s|)\leq\mathrm{E}[|B(t)-B(s)|^{2}]\leq\ell\gamma
^{2}(|t-s|).
\]

In this section we find a sharp estimate of the probability for $B$ to hit
points (note that $d=1$). Since $B$ is almost surely continuous and Gaussian,
we know that there is a positive probability to hit any point $z\in\mathbb{R}$
in any time interval $[a,b]$ with $0<a<b$. We have the following quantitative estimate.

\begin{prop}
\label{hitd1} There exists a universal positive constant $c_{u}$ and a
positive constant $t_{0}$ depending only on the law of $B$, such that for all
$z\in\mathbb{R}$, for all $a,b$ such that $0<a<b$ and $b-a\leq t_{0}$,%
\[
\P (B\left(  [a,b]\right)  \ni z)\leq\frac{c_{u}\sqrt{\ell}}{\gamma\left(
a\right)  }f_{\gamma}\left(  b-a\right)  .
\]
where $\ell$ is the constant in \textnormal{(\ref{e1})}, and
\[
f_{\gamma}\left(  x\right)  :=\gamma\left(  x\right)  \sqrt{\log2}+\int
_{0}^{1/2}\gamma\left(  xy\right)  \frac{dy}{y\sqrt{\log\left(  1/y\right)  }%
}.
\]

\end{prop}

In any specific situation, the function $f_{\gamma}$ can be computed more
explicitly. In most cases, both terms in $f_{\gamma}$ are commensurate. We
state such a situation as follows.

\begin{cor}
\label{hitd1cor}Assume there exists $k,y_{0}>0$ such that
\begin{equation}
\int_{0}^{1/2}\gamma\left(  xy\right)  \frac{dy}{y\sqrt{\log\left(
1/y\right)  }}\leq k\gamma\left(  x\right)  \label{gammamult}%
\end{equation}
for all $x\in\lbrack0,y_{0}]$. Then, for some constant $L$ depending only on
$\gamma$ and $y_{0}$, for all $z\in\mathbb{R}$ and for all $a,b$ such that
$0<a<b$ and $b-a\leq t_{0}$,
\[
\P (B\left(  [a,b]\right)  \ni z)\leq\frac{L\sqrt{\ell}}{\gamma\left(
a\right)  }\gamma\left(  b-a\right)  .
\]

\end{cor}

The corollary's assumption (\ref{gammamult}) is satisfied, and the corollary's
conclusion holds, for instance, for $\gamma\left(  r\right)  =\gamma_{H,\beta
}\left(  r\right)  :=r^{H}\log^{\beta}\left(  1/r\right)  $ for some $\beta
\in\mathbb{R}$ and $H\in(0,1)$, or $\beta\leq0$ and $H=1$. We will see what
this implies in Section 5. In examples such as these, particularly when
$\beta=0$, i.e. $\gamma\left(  r\right)  =r^{H}$, we can use the same method
of proof to show that up to a multiplicative constant not dependent on $b-a$,
$\gamma\left(  b-a\right)  $ is also a lower bound for $\P (B\left(
[a,b]\right)  \ni z)$, justifying our terminology of \textquotedblleft
sharp\textquotedblright\ mentioned above. Since this lower bound would not be
used in this article, we omit further discussion of how to establish it.

The class of examples where $\gamma=\gamma_{H,\beta}$ is a special case of
continuous functions $\gamma$ with a positive \textquotedblleft
index\textquotedblright: $\mathrm{ind}f:=\inf\left\{  \alpha>0:f\left(
x\right)  =o\left(  x^{\alpha}\right)  \right\}  $. Saying that a function $f$
has a positive finite index thus simply means that it is negligible w.r.t.
some power function but not all polynomials. The index of a function is a
concept which was already used for estimating hitting probabilities in
\cite{Bierme:09} in the context of $\gamma_{H,0}$ (the case $\beta=0$) (see
also \cite{Testard:86}). More generally, the index of $\gamma_{H,\beta}$ is
equal to $H$ no matter what $\beta$ is. However, being able to compute the
index of a function $\gamma$ is not enough to be able to estimate its hitting
probabilities precisely, as the last corollary above, or as the results in
Section 5 show, since these results depend on the values of $\beta$.
Nevertheless, we record and sketch a proof of the reassuring fact that any
function $\gamma$ we might use with positive finite index satisfies condition
\ref{gammamult}.

\begin{lemma}
\label{Lemindex}Define \textrm{ind}$\gamma:=\sup\left\{  \alpha>0:\gamma
\left(  x\right)  =o\left(  x^{\alpha}\right)  \right\}  $. Assume $\gamma$ is
continuous and increasing. Assume \textrm{ind}$\gamma\in(0,\infty)$. Then
$\gamma$ satisfies condition \textnormal{(\ref{gammamult})}.
\end{lemma}

\begin{proof}
Let $\alpha=\mathrm{ind}\gamma$. The lemma's assumption implies that for any
$\varepsilon>0$, there exists a constant $c$ such that if $x\in\lbrack0,1/2]$,
$\gamma\left(  x\right)  \leq cx^{\alpha-\varepsilon}$. It also implies that
there exists $C>0$ and a sequence $\left(  x_{n}\right)  _{n}$ decreasing to
$0$ such that $\gamma\left(  x_{n}\right)  \geq C\left(  x_{n}\right)
^{\alpha+\varepsilon}$. To shorten this sketch of proof, we will assume that
$\gamma\left(  x\right)  \geq Cx^{\alpha+\varepsilon}$ holds for all
$x\in\lbrack0,1/2]$; the details in the general case are left to the reader.
We now only need to show that%
\[
I:=\int_{0}^{1/2}\frac{\gamma\left(  xy\right)  }{\gamma\left(  x\right)
}\frac{dy}{y\sqrt{\log\left(  1/y\right)  }}%
\]
is bounded as $x$ approaches the origin. For $x<1/2$, with constants $k$ which
may change from line to line, using the bounds on $\gamma$ and the fact that
$\gamma$ is increasing, we have%
\begin{align*}
0  &  \leq I=\int_{0}^{x}\frac{\gamma\left(  xy\right)  }{\gamma\left(
x\right)  }\frac{dy}{y\sqrt{\log\left(  1/y\right)  }}+\int_{x}^{1/2}%
\frac{\gamma\left(  xy\right)  }{\gamma\left(  x\right)  }\frac{dy}%
{y\sqrt{\log\left(  1/y\right)  }}\\
&  \leq kx^{-2\varepsilon}\int_{0}^{x}y^{\alpha-\varepsilon-1}dy+k\int
_{x}^{1/2}\frac{dy}{y\sqrt{\log\left(  1/y\right)  }}\frac{\gamma\left(
x^{2}\right)  }{\gamma\left(  x\right)  }\\
&  \leq kx^{\alpha-3\varepsilon}+kx^{\alpha-3\varepsilon}\sqrt{\log\left(
1/x\right)  }.
\end{align*}
By choosing a value $\varepsilon\in(0,\alpha/3)$, the result follows.
\end{proof}

For processes which are continuous but not H\"{o}lder-continuous
(\ref{gammamult}) may fail, as may the above corollary's conclusion. For
instance, this occurs with $\gamma=\gamma_{0,\beta}$. To guarantee that $B$ is
continuous in that particular case, we need only assume $\beta<-1/2$. In this
case, we get a $\sqrt{\log}$ correction: for $\gamma\left(  r\right)
=\log^{\beta}\left(  1/r\right)  $,%
\[
\P (B\left(  [a,b]\right)  \ni z)\leq\frac{2c_{u}\sqrt{\ell}}{\gamma\left(
a\right)  }\gamma\left(  b-a\right)  \sqrt{\log(\frac{1}{b-a})},
\]
as shown via a direct estimation of $f_{\gamma}$ which we omit here.

\begin{proof}
[Proof of Proposition \ref{hitd1}]~\vspace*{0.1in}

\emph{Step 1: setup.} The event we must estimate is $A=\left\{  B\left(
[a,b]\right)  \ni z\right\}  $. We clearly have%
\[
A=\left\{  \min_{[a,b]}B\leq z\leq\max_{\lbrack a,b]}B\right\}  .
\]
Now consider
\begin{align*}
A_{a,z}  &  :=\left\{  B\left(  a\right)  \leq z\leq\max_{\lbrack
a,b]}B\right\}  ,\\
A_{b,z}  &  :=\left\{  B\left(  b\right)  \leq z\leq\max_{\lbrack
a,b]}B\right\}
\end{align*}
Since $B\left(  a\right)  \geq\min_{[a,b]}B$, we have $A_{a,z}\subset A$, and
similarly $A_{b,z}\subset A$. Let us prove that%
\begin{equation}
\P \left(  A\right)  \leq\P \left(  A_{a,z}\right)  +\P \left(  A_{b,z}%
\right)  +\P \left(  A_{a,-z}\right)  +\P \left(  A_{b,-z}\right)  .
\label{4P}%
\end{equation}
To lighten the notation, set $m=\min_{[a,b]}B$, $M=\max_{[a,b]}B$. Now
consider the case $B\left(  a\right)  <B\left(  b\right)  $, we get $m\leq
B\left(  a\right)  \leq B\left(  b\right)  \leq M$. Thus the interval $[m,M]$
is the union of the three intervals $I=[m,B\left(  a\right)  ]$, $J=[B\left(
a\right)  ,B\left(  b\right)  ]$, and $K=[B\left(  b\right)  ,M]$. In the
event $A\cap\left\{  B\left(  a\right)  <B\left(  b\right)  \right\}  $, if
$z\in I$, then $z\in\lbrack m,B\left(  b\right)  ]$; if $z\in J$, then
$z\in\lbrack B\left(  a\right)  ,M]$, and finally if $z\in K$, then
$z\in\lbrack B\left(  a\right)  ,M]$ also. This proves that%
\[
\left[  A\cap\left\{  B_{a}<B_{b}\right\}  \right]  \subset\left[  \left\{
B\left(  a\right)  \leq z\leq M\right\}  \cup\left\{  m\leq z\leq B\left(
b\right)  \right\}  \right]  .
\]
By reversing the roles of $B\left(  a\right)  $ and $B\left(  b\right)  $ we
get%
\[
\left[  A\cap\left\{  B_{b}<B_{a}\right\}  \right]  \subset\left[  \left\{
B\left(  b\right)  \leq z\leq M\right\}  \cup\left\{  m\leq z\leq B\left(
a\right)  \right\}  \right]  .
\]
We immediately get
\begin{align*}
\P \left(  A\right)   &  \leq\P \left(  A\cap\left\{  B_{a}<B_{b}\right\}
\right)  +\P \left(  A\cap\left\{  B_{b}<B_{a}\right\}  \right) \\
&  \leq\P \left[  \left\{  B\left(  a\right)  \leq z\leq M\right\}
\cup\left\{  m\leq z\leq B\left(  b\right)  \right\}  \right] \\
&  +\P \left[  \left\{  B\left(  b\right)  \leq z\leq M\right\}  \cup\left\{
m\leq z\leq B\left(  a\right)  \right\}  \right] \\
&  \leq\P \left[  B\left(  a\right)  \leq z\leq M\right]  +\P \left[  m\leq
z\leq B\left(  b\right)  \right] \\
&  +\P \left[  B\left(  b\right)  \leq z\leq M\right]  +\P \left[  m\leq z\leq
B\left(  a\right)  \right]  .
\end{align*}
The first and third terms in the last sum of four terms above are precisely
$\P \left(  A_{a,z}\right)  $ and $\P \left(  A_{b,z}\right)  $. For the
second term, we can write
\[
\left\{  m\leq z\leq B\left(  b\right)  \right\}  =\left\{  \max
_{[a,b]}\left(  -B\right)  \geq-z\geq-B\left(  b\right)  \right\}  ,
\]
and since $B$ and $-B$ have the same law, the second term is equal to
$\P \left(  A_{b,-z}\right)  $. Similarly the fourth term is equal to
$\P \left(  A_{a,-z}\right)  $, proving our claim (\ref{4P}).\vspace*{0.1in}

\emph{Step 2: Gaussian linear regression}. Here we appeal to the linear
regression between $B\left(  a\right)  $ and any value $B\left(  s\right)  $
for $s\in(a,b]$. Recall the notation
\[
\sigma^{2}\left(  s,t\right)  :=\mathrm{E}\left[  B\left(  s\right)  B\left(
t\right)  \right]  ,
\]
$\delta^{2}\left(  s,t\right)  :=\mathrm{E}\left[  \left(  B\left(  s\right)
-B\left(  t\right)  \right)  ^{2}\right]  $, and that $\gamma^{2}\left(
t\right)  =\mathrm{Var}B\left(  t\right)  $. For any $s\in\lbrack a,b]$, let
\[
\rho\left(  s\right)  :=\frac{\sigma^{2}\left(  a,s\right)  }{\gamma
^{2}\left(  a\right)  }.
\]
It is elementary that there exists a centered Gaussian random variable
$R\left(  s\right)  $ independent of $B\left(  a\right)  $ such that%
\[
B\left(  s\right)  =\rho\left(  s\right)  B\left(  a\right)  +R\left(
s\right)  .
\]
Note that this defines $R\left(  s\right)  :=B\left(  s\right)  -\rho\left(
s\right)  B\left(  a\right)  $, so that $R$ is almost-surely continuous on
$[a,b]$. It is also elementary to compute the covariance structure of
$R\left(  s\right)  $:%
\begin{align*}
\mathrm{E}\left[  \left(  R\left(  s\right)  -R\left(  t\right)  \right)
^{2}\right]   &  =\mathrm{E}\left[  \left(  B\left(  s\right)  -B\left(
t\right)  \right)  ^{2}\right]  +\left(  \rho\left(  s\right)  -\rho\left(
t\right)  \right)  ^{2}\mathrm{E}\left[  B\left(  a\right)  ^{2}\right] \\
&  -2\left(  \rho\left(  s\right)  -\rho\left(  t\right)  \right)
\mathrm{E}\left[  B\left(  a\right)  \left(  B\left(  s\right)  -B\left(
t\right)  \right)  \right] \\
&  =\delta^{2}\left(  s,t\right)  +\frac{1}{\gamma^{2}\left(  a\right)
}\left(  \sigma^{2}\left(  a,s\right)  -\sigma^{2}\left(  a,t\right)  \right)
^{2}\\
&  -2\frac{1}{\gamma^{2}\left(  a\right)  }\left(  \sigma^{2}\left(
a,s\right)  -\sigma^{2}\left(  a,t\right)  \right)  \left(  \sigma^{2}\left(
a,s\right)  -\sigma^{2}\left(  a,t\right)  \right) \\
&  =\delta^{2}\left(  s,t\right)  -\frac{1}{\gamma^{2}\left(  a\right)
}\left(  \sigma^{2}\left(  a,s\right)  -\sigma^{2}\left(  a,t\right)  \right)
^{2}.
\end{align*}
\vspace*{0.1in}

\emph{Step 3: Estimation of }$\P \left(  A_{a,z}\right)  \emph{\ via}$ $R$. We
can now write%
\begin{align*}
\P \left(  A_{a,z}\right)   &  =\P \left(  B\left(  a\right)  \leq z\leq
\max_{s\in\lbrack a,b]}\left\{  \rho\left(  s\right)  B\left(  a\right)
+R\left(  s\right)  \right\}  \right) \\
&  \leq\P \left(  B\left(  a\right)  \leq z\leq\max_{s\in\lbrack a,b]}\left\{
\rho\left(  s\right)  B\left(  a\right)  \right\}  +\max_{s\in\lbrack
a,b]}R\left(  s\right)  \right) \\
&  =\P \left(  B\left(  a\right)  \leq z\leq B\left(  a\right)  +\max
_{s\in\lbrack a,b]}\left\{  \left(  \rho\left(  s\right)  -1\right)  B\left(
a\right)  \right\}  +\max_{s\in\lbrack a,b]}R\left(  s\right)  \right) \\
&  \leq\P \left(  B\left(  a\right)  \leq z\leq B\left(  a\right)  +B\left(
a\right)  \mathrm{sgn}\left(  B\left(  a\right)  \right)  \max_{s\in\lbrack
a,b]}\left\vert \rho\left(  s\right)  -1\right\vert +\max_{s\in\lbrack
a,b]}R\left(  s\right)  \right) \\
&  =\P \left(  z-\frac{\max_{[a,b]}R}{1+\mathrm{sgn}\left(  B\left(  a\right)
\right)  \max_{s\in\lbrack a,b]}\left\vert \rho\left(  s\right)  -1\right\vert
}\leq B\left(  a\right)  \leq z\right)  ,
\end{align*}
where the last inequality holds provided that $\max_{s\in\lbrack
a,b]}\left\vert \rho\left(  s\right)  -1\right\vert <1$, also noting that
$\max_{[a,b]}R\geq0$ since $R\left(  a\right)  =0$. Since $\sigma^{2}\left(
a,\cdot\right)  $ is continuous on $[a,b]$, for $b-a$ small enough,
$\sigma^{2}\left(  a,s\right)  $ can be made arbitrarily close to $\sigma
^{2}\left(  a,a\right)  =\gamma^{2}\left(  a\right)  $. Thus, by definition of
$\rho$, we can make $\rho\left(  s\right)  $ close to $1$. Hence, there exists
$t_{0}>0$ such that $0<b-a<t_{0}$ implies $\max_{s\in\lbrack a,b]}\left\vert
\rho\left(  s\right)  -1\right\vert \leq1/2$, which implies in turn that%
\[
\frac{2}{3}\leq\frac{1}{1+\mathrm{sgn}\left(  B\left(  a\right)  \right)
\max_{s\in\lbrack a,b]}\left\vert \rho\left(  s\right)  -1\right\vert }\leq2.
\]
Consequently,
\begin{align*}
\P \left(  A_{a,z}\right)   &  \leq\P \left(  z-2\max_{[a,b]}R\leq B\left(
a\right)  \leq z\right) \\
&  =\int_{(z-2\max_{[a,b]}R)/\gamma\left(  a\right)  }^{z/\gamma\left(
a\right)  }\frac{1}{\sqrt{2\pi}}\exp\left(  -\frac{x^{2}}{2}\right)  dx\\
&  \leq\frac{2}{\gamma\left(  a\right)  \sqrt{2\pi}}\mathrm{E}\left[
\max_{[a,b]}R\right]  .
\end{align*}
In the last equality above, we used the independence of the process $R$ and
the random variable $B\left(  a\right)  $.\vspace*{0.1in}

\emph{Step 4: Estimation of }$\mathrm{E}\left[  \max_{[a,b]}R\right]  $. We
saw in Step 2 that $R$ is a centered Gaussian process with a canonical metric
$\mathrm{E}^{1/2}\left[  \left(  R\left(  s\right)  -R\left(  t\right)
\right)  ^{2}\right]  $ bounded above by $\delta\left(  s,t\right)  $, the
canonical metric of $B$. From our assumption (\ref{e1}), we have
$\delta\left(  s,t\right)  \leq\ell\gamma\left(  |t-s|\right)  $. This means
that, to cover the interval $[a,b]$ with balls of radius $x$ in the canonical
metric of $R$, we require no more than $N\left(  x\right)  :=\left(
b-a\right)  /\gamma^{-1}\left(  x/\sqrt{\ell}\right)  $ such balls. Since
$\gamma$ is increasing, the diameter of $[a,b]$ in this canonical metric is
bounded above by $\sqrt{\ell}\gamma\left(  b-a\right)  $. We can thus apply
the classical entropy upper bound of R. Dudley (see \cite{Adler:90}) to
obtain, for some universal constant $C_{\mathrm{univ}}$,
\begin{equation}
\label{FromDudley}%
\begin{split}
&  \frac{1}{C_{\mathrm{univ}}}\mathrm{E}\left[  \max_{[a,b]}R\right]  \leq
\int_{0}^{\sqrt{k}\gamma\left(  b-a\right)  }\sqrt{\log N\left(  x\right)
}dx\\
&  \qquad=\int_{0}^{\sqrt{k}\gamma\left(  b-a\right)  }\sqrt{\log\frac
{b-a}{\gamma^{-1}\left(  x/\sqrt{k}\right)  }}dx\\
&  \qquad=\sqrt{k}\int_{0}^{b-a}\sqrt{\log\frac{b-a}{r}}d\gamma\left(
r\right) \\
&  \qquad\leq\sqrt{k}\int_{0}^{\left(  b-a\right)  /2}\sqrt{\log\frac{b-a}{r}%
}d\gamma\left(  r\right)  +\sqrt{k}\left(  \gamma\left(  b-a\right)
-\gamma\left(  \frac{b-a}{2}\right)  \right)  \sqrt{\log2}%
\end{split}
\end{equation}
where we used a change of variables. By an integration by parts and another
change of variables, using the fact that since $B$ is a.s. continuous, we have
$\gamma\left(  x\right)  =o\left(  1/\sqrt{\log\left(  1/x\right)  }\right)
$, we get
\begin{equation}
\int_{0}^{\left(  b-a\right)  /2}\sqrt{\log\frac{b-a}{r}}d\gamma\left(
r\right)  =\gamma\left(  (b-a)/2\right)  \sqrt{\log2}+\int_{0}^{1/2}%
\gamma\left(  (b-a)y\right)  \frac{dy}{y\sqrt{\log\left(  1/y\right)  }}.
\label{EntropyCalc}%
\end{equation}
Relations (\ref{FromDudley}) and (\ref{EntropyCalc}) now yield%
\[
\mathrm{E}\left[  \max_{[a,b]}R\right]  \leq C_{\mathrm{univ}}\sqrt{\ell
}\left(  \gamma\left(  b-a\right)  \sqrt{\log2}+\int_{0}^{1/2}\gamma\left(
(b-a)y\right)  \frac{dy}{y\sqrt{\log\left(  1/y\right)  }}\right)  ,
\]
where we recognize the function $f_{\gamma}$ identified in the statement of
the proposition.\vspace{0.1in}

\emph{Step 5: Conclusion}. The results of Step 3 and Step 4 now imply, for
another universal constant $C_{\mathrm{univ}}^{\prime}$%
\begin{equation}
\P \left(  A_{a,z}\right)  \leq\frac{C_{\mathrm{univ}}^{\prime}}{\gamma\left(
a\right)  }f_{\gamma}\left(  b-a\right)  . \label{PAaz}%
\end{equation}
There is nothing in the arguments of Steps 2,3, and 4 which prevents us from
relating everything to $B\left(  b\right)  $ rather than $B\left(  a\right)
$. The only quantitative differences occur as follows: with $R$ relative to
$b$, not $a$,

\begin{itemize}
\item at the end of Step 2, the expression for $\mathrm{E}\left[  \left(
R\left(  s\right)  -R\left(  t\right)  \right)  ^{2}\right]  $ involved $b$,
not $a$, but this is still bounded above by $\delta^{2}\left(  s,t\right)  $,
so there is no quantitative change in the application of Step 2 in Step 4,
i.e. the upper bound of Step 4 remains unchanged for the new $R$;

\item the expression for $\rho_{s}$ in\ Step 2 is now relative to $b$ rather
than $a$, but we can still have $\max_{s\in\lbrack a,b]}\left\vert \rho\left(
s\right)  -1\right\vert \leq1/2$ for $b-a$ small enough;

\item at the end of Step 3, we now get%
\begin{align*}
\P \left(  A_{b,z}\right)   &  \leq\P \left(  z-2\max_{[a,b]}R\leq B\left(
b\right)  \leq z\right) \\
&  =\int_{(z-2\max_{[a,b]}R)/\gamma\left(  b\right)  }^{z/\gamma\left(
b\right)  }\frac{1}{\sqrt{2\pi}}\exp\left(  -\frac{x^{2}}{2}\right)  dx\\
&  \leq\frac{2}{\gamma\left(  b\right)  \sqrt{2\pi}}\mathrm{E}\left[
\max_{[a,b]}R\right]  .
\end{align*}

\end{itemize}

Consequently, we have proved%
\begin{equation}
\P \left(  A_{b,z}\right)  \leq\frac{C_{\mathrm{univ}}^{\prime}}{\gamma\left(
b\right)  }f_{\gamma}\left(  b-a\right)  . \label{PAbz}%
\end{equation}
Since the estimates (\ref{PAaz}) and (\ref{PAbz}) are uniform in $z$, they
also apply to $\P \left(  A_{a,-z}\right)  $ and $\P \left(  A_{b,-z}\right)
$ respectively. Since $\gamma\left(  b\right)  >\gamma\left(  a\right)  $,
(\ref{4P}) now implies the statement of the proposition, with $c_{u}=\left(
8/\sqrt{2\pi}\right)  C_{\mathrm{univ}}$.
\end{proof}

\section{Upper Hausdorff measure bound for the hitting probabilities}

Recall that $B=(B(t),t\in\mathbb{R}_{+})$ is a centered continuous Gaussian
process in $\mathbb{R}$ with variance $\gamma^{2}(t)$ for some continuous
strictly increasing function $\gamma:\mathbb{R}_{+}\rightarrow\mathbb{R}_{+}$
with $\lim_{0}\gamma=0$, such that for some constant $\ell\geq1$ and for all
$s,t\geq0$,
\[
(1/\ell)\gamma^{2}(|t-s|)\leq\mathrm{E}[|B(t)-B(s)|^{2}]\leq\ell\gamma
^{2}(|t-s|).
\]

The aim of this section is to prove an upper bound for the probability that a
vector of $d$ iid copies of $B$, also denoted by $B$, hits a set
$A\subset\mathbb{R}^{d}$, in terms of a certain Hausdorff measure of $A$.

For a function $\varphi:\mathbb{R}_{+}\rightarrow\mathbb{R}_{+}$,
right-continuous and non-decreasing near zero with $\lim_{0+}\varphi=0$, we
define the $\varphi$-Hausdorff measure of a set $A\subset\mathbb{R}^{d}$ as
\[
{\mathcal{H}}_{\varphi}(A)=\lim_{\varepsilon\rightarrow0^{+}}\inf\left\{
\sum_{i=1}^{\infty}\varphi(2r_{i}):A\subseteq\bigcup_{i=1}^{\infty}%
B(x_{i},r_{i}),\ \sup_{i\geq1}r_{i}\leq\varepsilon\right\}  ,
\]
where $B(x_{i},r_{i})$ denotes the open ball of center $x_{i}$ and radius
$r_{i}$ in $\mathbb{R}^{d}$. When $\varphi(r)=r^{\beta}$, $\beta>0$, we write
${\mathcal{H}}_{\beta}$ and call it the $\beta$-Hausdorff measure. When
$\beta\leq0$, we define $\mathcal{H}_{\beta}$ to be infinite.

We first provide an upper bound for the probability that $B$ hits a small ball
in $\mathbb{R}^{d}$. The papers \cite[Lemma 7.8]{Xiao:09} or \cite[Lemma
3.1]{Bierme:09} contain some analogous results in the case that $\gamma$ is a
power, but the techniques there do not seem to be applicable to the case of
general $\gamma$.

\begin{prop}
\label{sma} For all $0<a<b<\infty$, with $b-a$ small enough, for all
$z\in\mathbb{R}^{d}$, $\varepsilon>0$,
\[
\P (B([a,b])\cap B(z\,,\varepsilon)\neq\emptyset)\leq\left(  \varepsilon
\frac{2\kappa\gamma\left(  b\right)  }{\gamma^{2}\left(  a\right)  }\left(
1+\frac{1}{F\left(  \left\vert z\right\vert \right)  }\right)  +\frac
{c_{u}\sqrt{\ell}}{\gamma\left(  a\right)  }f_{\gamma}\left(  b-a\right)
\right)  ^{d}.
\]
where $B(z,\varepsilon)$ denotes the open ball of center $z$ and radius
$\varepsilon$ in $\mathbb{R}^{d}$, $\ell$ and $\gamma$ are from condition
\textnormal{(\ref{e1})}, $c_{u}$ and $f_{\gamma}$ are defined in Proposition
\text{\ref{hitd1}},%
\begin{align*}
F\left(  z\right)   &  =\left\{
\begin{array}
[c]{r}%
1\ \mathrm{for~}z\leq\gamma\left(  b\right)  ,\\
1-e^{-2\mathrm{arctanh}\left(  \gamma^{2}\left(  b\right)  /z^{2}\right)
}\ \mathrm{for~}z>\gamma\left(  b\right)  ,
\end{array}
\right. \\
\kappa &  =\P \left[  \inf_{[a,b]}B>\gamma\left(  b\right)  \right]  .
\end{align*}

In particular, when Condition \textnormal{(\ref{gammamult})} is satisfied, with
$L$ as in Corollary \text{\ref{hitd1cor}},
\[
\P (B([a,b])\cap B(z\,,\varepsilon)\neq\emptyset)\leq\left(  \varepsilon
\frac{2\kappa\gamma\left(  b\right)  }{\gamma^{2}\left(  a\right)  }\left(
1+\frac{1}{F\left(  \left\vert z\right\vert \right)  }\right)  +\frac
{L\sqrt{\ell}}{\gamma\left(  a\right)  }\gamma\left(  b-a\right)  \right)
^{d}.
\]

\end{prop}

\begin{rmk}
Note that $\mathrm{arctanh}\left(  x\right)  $ is equivalent to $x$ for $x$
small, and therefore, for large $z$, $1/F\left(  z\right)  $ in the above
proposition is equivalent to $z^{2}/\left(  2\gamma^{2}\left(  b\right)
\right)  $.
\end{rmk}

\begin{rmk}
\label{d1}Since the various components of $B$ are independent, it is
equivalent to prove this proposition with $d=1$ only.
\end{rmk}

The proof of this proposition needs some preparation. We first explain why the
results of Section \ref{hitd1sec} on hitting points in one dimension will be
needed to prove this proposition. The event whose probability we need to
estimate is%
\begin{align*}
D  &  :=\left\{  B([a,b])\cap B(z\,,\varepsilon)\neq\emptyset\right\}
=\left\{  \inf_{s\in\lbrack a,b]}\left\vert B\left(  s\right)  -z\right\vert
\leq\varepsilon\right\} \\
&  =\left\{  0<\inf_{s\in\lbrack a,b]}\left\vert B\left(  s\right)
-z\right\vert \leq\varepsilon\right\}  \cup\left\{  B\left(  [a,b]\right)  \ni
z\right\}  =:D_{1}\cup D_{2}.
\end{align*}
The second event $D_{2}$ in the last line above is the one whose probability
we estimated in Section \ref{hitd1sec}. We choose to separate it from the
remaining event $D_{1}$ because, since $B$ hits points with positive
probability during $[a,b]$, the random variable $Z:=\inf_{s\in\lbrack
a,b]}\left\vert B\left(  s\right)  -z\right\vert $ has an atom at $0$. Note
that $D_{1}$ and $D_{2}$ are disjoint. In any case, Proposition \ref{hitd1}
shows that to prove Proposition \ref{sma}, it is sufficient to establish%
\begin{equation}
\P \left(  D_{1}\right)  =\P \left(  0<Z\leq\varepsilon\right)  \leq
C\varepsilon\label{D1}%
\end{equation}
for the appropriate constant $C$. To prove this, it would be sufficient to
show that, $Z$ has a bounded density on $(0,+\infty)$.

To establish such a density bound, we must take a minor detour via the
Malliavin calculus. A criterion was established in \cite[Theorem
3.1]{Nourdin:09} for proving the existence of a density and a method for
estimating it quantitatively. That technique does not apply to random
variables with atoms, but in our case, $Z$ has a single atom, at the point $0$
at the edge of its support, and we are able to adapt the method of
\cite{Nourdin:09} to such a situation.

The needed elements of Malliavin calculus are the following. Details can be
found in \cite[Chapters 2 and 10]{Nourdin:12}. Let $D$ be the Malliavin
derivative operator in the Wiener space $L^{2}\left(  \Omega,\mathcal{F}%
,\P \right)  $ induced by the process $B$. Let $\mathbb{D}^{1,2}$ be the
Gross-Sobolev subspace of $L^{2}\left(  \Omega\right)  $, identified with the
domain of $D$. Let $X$ be a centered random variable in $\mathbb{D}^{1,2}$ and
define
\begin{equation}
G_{X}:=\int_{0}^{+\infty}due^{-u}\mathrm{E}\left[  \left.  \int_{\mathbb{R}%
_{+}}D_{r}X~D_{r}^{\left(  u\right)  }X^{\left(  u\right)  }dr\right\vert
\mathcal{F}\right]  , \label{GX}%
\end{equation}
and%
\begin{equation}
g_{X}\left(  x\right)  =\mathrm{E}\left[  G_{X}~|~X=x\right]  , \label{gx}%
\end{equation}
where the notation $X^{\left(  u\right)  }$ denotes a random variable with the
same law as $X$, but constructed using a copy $B^{u}$ of the Gaussian field
$B$ such that the correlation coefficient $\mathrm{Corr}\left(  B,B^{u}%
\right)  =e^{-u}$, and $D^{\left(  u\right)  }$ is the Malliavin derivative
operator on the Wiener space induced by $B^{u}$. The expression in (\ref{GX})
coincides with the random variable $\left\langle DX,-DL^{-1}X\right\rangle
_{B}$ where $L^{-1}$ is the pseudo-inverse of the generator of the
Ornstein-Uhlenbeck semigroup on $B$'s Wiener space, and $\left\langle
\cdot,\cdot\right\rangle _{B}$ is the canonical inner product defined by $B$'s
Gaussian Wiener integrals; this coincidence comes from the so-called Mehler
formula, and we also have that for every continuously differentiable function
$f$ with bounded derivative,
\begin{equation}
\mathrm{E}\left[  Xf\left(  X\right)  \right]  =\mathrm{E}\left[  g_{X}\left(
X\right)  f^{\prime}\left(  X\right)  \right]  . \label{NP}%
\end{equation}
In the sequel, we will only need to use (\ref{GX}), (\ref{gx}), and
(\ref{NP}). We have the following

\begin{prop}
\label{dens}Let $X$ be a centered random variable in $\mathbb{D}^{1,2}$, and
$G_{X}$ and $g_{X}$ be defined in \textnormal{(\ref{GX})}, \textnormal{(\ref{gx})}.
The support of the law of $X$ is an interval $[\alpha,\beta]$ with
$-\infty\leq\alpha<0<\beta\leq+\infty$. Assume there exists $\alpha^{\prime
}\in(\alpha,0)$ such that $g_{X}\left(  x\right)  >0$ for all $x\in
\lbrack\alpha^{\prime},\beta)$. Then $X$ has a density $\rho$ on
$[\alpha^{\prime},\beta)$, and for almost every $z\in\lbrack\alpha^{\prime
},\beta)$,%
\begin{equation}
\rho\left(  x\right)  =\frac{\mathrm{E}\left[  \left\vert X\right\vert
\right]  }{2g_{X}\left(  x\right)  }\exp\left(  -\int_{0}^{x}ydy/g_{X}\left(
y\right)  \right)  . \label{rhogee}%
\end{equation}

\end{prop}

\begin{proof}
The proof of this proposition varies only slightly from that of \cite[Theorem
3.1]{Nourdin:09}; we provide it here for completeness. The statement about the
support of $X$ is well-known \cite[Proposition 2.1.7]{Nualart:06}. Let $A$ be
a Borel set included in $[\alpha^{\prime},\beta)$, and assume that its
Lebesgue measure is $0$. By using a monotone approximation argument, we can
apply (\ref{NP}) to $f\left(  x\right)  =\int_{\alpha^{\prime}}^{x}%
\mathbf{1}_{A}\left(  y\right)  dy$. Thus
\[
0=\mathrm{E}\left[  Xf\left(  X\right)  \right]  =\mathrm{E}\left[
\mathbf{1}_{A}\left(  X\right)  g_{X}\left(  X\right)  \right]  .
\]
Since $A\subset\lbrack\alpha^{\prime},\beta)$, by assumption, $g_{X}\left(
X\right)  >0$ on the event $\left\{  X\in A\right\}  $. Consequently,
$\mathbf{1}_{A}\left(  X\right)  =0$ almost surely, i.e. $\P \left[  X\in
A\right]  =0$, which means the law of $X$ restricted to $[\alpha^{\prime
},\beta)$ is absolutely continuous w.r.t. Lebesgue's measure, and therefore
$X$ has a density $\rho$ on $[\alpha^{\prime},\beta)$, and note that
$\rho\left(  X\right)  $ is positive almost surely.

Now for any continuous function $f$ with compact support in $[\alpha^{\prime
},\beta)$, and its anti-derivative $F\left(  x\right)  =\int_{\alpha^{\prime}%
}^{x}f\left(  y\right)  dy$ (which is necessarily bounded), by (\ref{NP}) we
have
\[
\mathrm{E}\left[  g_{X}\left(  X\right)  f\left(  X\right)  \right]
=\mathrm{E}\left[  XF\left(  X\right)  \right]  =\int_{\alpha^{\prime}}%
^{\beta}\rho\left(  y\right)  yF\left(  y\right)  dy.
\]
We perform the integration by parts with parts $F\left(  y\right)  $ and
$\rho\left(  y\right)  ydy$. Note that $\varphi$ is differentiable almost
everywhere on $[\alpha^{\prime},\beta)$, and is bounded since $X\in
L^{2}\left(  \Omega\right)  \subset L^{1}\left(  \Omega\right)  $. Thus, with
$\varphi\left(  x\right)  =\int_{x}^{\beta}y\rho\left(  y\right)  dy$, we get%
\[
\mathrm{E}\left[  g_{X}\left(  X\right)  f\left(  X\right)  \right]
=\int_{\alpha^{\prime}}^{\beta}f\left(  y\right)  \varphi\left(  y\right)
dy+\lim_{x\rightarrow\beta}F\left(  x\right)  \varphi\left(  x\right)
-\lim_{x\rightarrow\alpha^{\prime}}F\left(  x\right)  \varphi\left(  x\right)
.
\]
Here, $\lim_{x\rightarrow\alpha^{\prime}}F\left(  x\right)  =0$, by
definition, and $\lim_{x\rightarrow\beta}\varphi\left(  x\right)  =0$ since
$X\in L^{1}\left(  \Omega\right)  $. Thus%
\[
\mathrm{E}\left[  g_{X}\left(  X\right)  f\left(  X\right)  \right]
=\int_{\alpha^{\prime}}^{\beta}f\left(  y\right)  \varphi\left(  y\right)
dy=\mathrm{E}\left[  \frac{\varphi\left(  X\right)  }{\rho\left(  X\right)
}f\left(  X\right)  \right]  .
\]
This implies that on the event $\left\{  X\in\lbrack\alpha^{\prime}%
,\beta)\right\}  $, $g_{X}\left(  X\right)  =\varphi\left(  X\right)
/\rho\left(  X\right)  $ almost surely, which, since $[\alpha^{\prime},\beta)$
is inside the support of $X$, implies that for almost every $x\in\lbrack
\alpha^{\prime},\beta)$, $g_{X}\left(  x\right)  =\varphi\left(  x\right)
/\rho\left(  x\right)  $.

Since by definition, $\varphi^{\prime}\left(  x\right)  =-x\rho\left(
x\right)  ,$we get an ordinary differential equation for $\varphi$, whose
unique solution is identical to the relation (\ref{rhogee}), provided one uses
the boundary condition given by $\varphi\left(  0\right)  $, which equals
$\mathrm{E}\left[  \left\vert X\right\vert \right]  /2$ because $\mathrm{E}%
\left[  X\right]  =0$.
\end{proof}

This proposition provides a convenient criterion to establish existence and
upper bounds on densities: if one can show that $g_{X}\left(  x\right)  \geq
c>0$ for all $x\in\lbrack\alpha^{\prime},\beta)$, then (\ref{rhogee}) implies
for all $x\in\lbrack\alpha^{\prime},\beta)$ :%
\begin{equation}
\rho_{X}\left(  x\right)  \leq\frac{\mathrm{E}\left[  \left\vert X\right\vert
\right]  }{2c}. \label{densUB}%
\end{equation}
This follows from the fact that $g_{X}$ is a positive function, so that for
any $x$, whether positive or negative, the exponential in the density formula
(\ref{rhogee}) is always less than $1$. The positivity of $g_{X}$ is
well-known (see \cite{Peccati:09}), and can also be inferred directly from
formula (\ref{rhogee}).

As it turns out, the random variable $Z$ is difficult to estimate via
Proposition \ref{dens}, because the expression one finds for $g_{Z-\mathrm{E}%
Z}$ via (\ref{GX}) is an integral of a signed function. However, an easy
expansion of $D_{1}$ is helpful. Since in $D_{1}$, $Z$ is positive, and $B$ is
continuous, this means that either $\inf_{s\in\lbrack a,b]}\left\vert B\left(
s\right)  -z\right\vert $ was attained for the whole trajectory $B\left(
[a,b]\right)  $ below the level $z$, or above it, and these two events are
disjoint. Therefore%
\begin{align}
\P \left(  D_{1}\right)   &  =\P \left(  0<\inf_{s\in\lbrack a,b]}\left(
B\left(  s\right)  -z\right)  _{+}\leq\varepsilon\right)  +\P \left(
0<\inf_{s\in\lbrack a,b]}\left(  B\left(  s\right)  -z\right)  _{-}%
\leq\varepsilon\right) \nonumber\\
&  =\P \left(  0<\inf_{s\in\lbrack a,b]}\left(  B\left(  s\right)  -z\right)
_{+}\leq\varepsilon\right)  +\P \left(  0<\inf_{s\in\lbrack a,b]}\left(
-B\left(  s\right)  -\left(  -z\right)  \right)  _{+}\leq\varepsilon\right)
\nonumber\\
&  =:D_{z}^{\prime}+D_{-z}^{\prime} \label{Deedee}%
\end{align}
where in the last line we used the fact that $B$ has a symmetric law.
According to the strategy leading to (\ref{densUB}), we only need to study the
random variable $G_{X}$ relative to
\begin{align*}
X  &  :=\inf_{s\in\lbrack a,b]}\left(  B\left(  s\right)  -z\right)  _{+}%
-\mu,\\
\mu &  :=\mathrm{E}\left[  \inf_{s\in\lbrack a,b]}\left(  B\left(  s\right)
-z\right)  _{+}\right]  .
\end{align*}

\begin{proof}
[Proof of Proposition \ref{sma}]~Recall that by Remark \ref{d1}, we assume $B$
is scalar.\vspace*{0.1in}

\emph{Step 0: what we must prove}. The centered random variable $X$ above is
supported in $[-\mu,+\infty)$. Moreover, it is a Lipshitz functional of a
continuous Gaussian process, and as such, belongs to $\mathbb{D}^{1,2}$. From
Proposition \ref{dens} and relation (\ref{densUB}), it is sufficient to prove
that there is a positive constant $c$ such that for any $x>-\mu$,
\[
g_{X}\left(  x\right)  \geq c.
\]
\vspace*{0.1in}

\emph{Step 1: Computing }$G_{X}$. To use formula (\ref{GX}), we must compute
$DX$. One may always assume that, with $\mathcal{H}$ the canonical Hilbert
space of the isonormal Gaussian process $W$ underlying $B$, for $s\in\lbrack
a,b]$, there exists an element $f_{s}\in\mathcal{H}$ such that $B\left(
s\right)  =W\left(  f_{s}\right)  $. Note that $\left\langle f_{s}%
,f_{t}\right\rangle =\sigma^{2}\left(  s,t\right)  :=\mathrm{E}\left[
B\left(  s\right)  B\left(  t\right)  \right]  $. Then, by using the same
argument as in the proof of \cite[Lemma 3.11]{Nourdin:09}, we find that on the
event $\left\{  X>-\mu\right\}  $,%
\[
DX=\mathbf{1}_{B\left(  \tau\right)  >z}f_{\tau},
\]
where $\tau=\arg\min_{s\in\lbrack a,b]}\left(  B\left(  s\right)  -z\right)
=\arg\min_{s\in\lbrack a,b]}B\left(  s\right)  $. Note that since $B$ is
continuous, this $\arg\min$ is uniquely defined in the event $\left\{
X>-\mu\right\}  $. Thus by the Mehler-type representation formula (\ref{GX}),%
\begin{align}
G_{X}  &  =\int_{0}^{\infty}du~e^{-u}\mathrm{E}\left[  \left.  \left\langle
\mathbf{1}_{B\left(  \tau\right)  >z}~f_{\tau};\mathbf{1}_{B^{\left(
u\right)  }\left(  \tau^{\left(  u\right)  }\right)  >z}~f_{\tau^{\left(
u\right)  }}\right\rangle \right\vert \mathcal{F}\right] \nonumber\\
&  =\int_{0}^{\infty}du~e^{-u}\mathrm{\tilde{E}}\left[  \mathbf{1}_{B\left(
\tau\right)  >z}\mathbf{1}_{B^{\left(  u\right)  }\left(  \tau^{\left(
u\right)  }\right)  >z}\left\langle f_{\tau};f_{\tau^{\left(  u\right)  }%
}\right\rangle \right] \nonumber\\
&  =\int_{0}^{\infty}du~e^{-u}\mathrm{\tilde{E}}\left[  \mathbf{1}_{B\left(
\tau\right)  >z}\mathbf{1}_{B^{\left(  u\right)  }\left(  \tau^{\left(
u\right)  }\right)  >z}\sigma^{2}\left(  \tau,\tau^{\left(  u\right)
}\right)  \right]  , \label{GXcomput}%
\end{align}
where $\mathrm{\tilde{E}}$ represents the expectation with respect to the
randomness in the independent copy $\tilde{B}$ of $B$, and the superscripts
$^{\left(  u\right)  }$ mean that the corresponding random variables are
relative to $B^{\left(  u\right)  }=e^{-u}B+\sqrt{1-e^{-2u}}\tilde{B}$%
.\vspace*{0.1in}

\emph{Step 2: Estimating }$g_{X}$. We must compute $\mathrm{E}\left[
G_{X}~|~X=x\right]  $ for any $x>-\mu$. Here, a convenient simplification
occurs: since we are conditioning by $\left\{  X=x\right\}  $, on this event,
$\inf_{s\in\lbrack a,b]}\left(  B\left(  s\right)  -z\right)  _{+}$ is
strictly positive, and in fact it equals $B_{\tau}-z$; therefore,
$\mathbf{1}_{B_{\tau}>z}=1$ almost surely on that event. Therefore, for any
$x>-\mu$,%
\[
g_{X}\left(  x\right)  =\int_{0}^{\infty}du~e^{-u}\mathrm{\tilde{E}E}\left[
\mathbf{1}_{B^{\left(  u\right)  }\left(  \tau^{\left(  u\right)  }\right)
>z}\sigma^{2}\left(  \tau,\tau^{\left(  u\right)  }\right)  ~|~X=x\right]  .
\]
The goal being to bound this expression uniformly from below, we note that
both $\tau$ and $\tau^{\left(  u\right)  }$ are in the non-random interval
$[a,b]$. Since $B$ is a.s. continuous, the bivariate function $\sigma^{2}$ is
uniformly continuous on $[a,b]\times\lbrack a,b]$. Since $a>0$, $\sigma
^{2}\left(  a,a\right)  =\gamma^{2}\left(  a\right)  >0$, and by making $b-a$
small enough, we can get $\min_{\left(  s,t\right)  \in\lbrack a,b]^{2}}%
\sigma^{2}\left(  s,t\right)  \geq\gamma^{2}\left(  a\right)  /2$. Thus%
\[
g_{X}\left(  x\right)  \geq\frac{\gamma^{2}\left(  a\right)  }{2}\int
_{0}^{\infty}du~e^{-u}\mathrm{\tilde{E}E}\left[  \mathbf{1}_{B^{\left(
u\right)  }\left(  \tau^{\left(  u\right)  }\right)  >z}~|~X=x\right]  .
\]
\vspace*{0.1in}

\emph{Step 3: Estimating the last expectation}. We now evaluate the remaining
expectation above. We have for any $x>-\mu$, and any $z\in\mathbb{R}$,%
\begin{align*}
&  \mathrm{\tilde{E}E}\left[  \mathbf{1}_{B^{\left(  u\right)  }\left(
\tau^{\left(  u\right)  }\right)  >z}~|~X=x\right] \\
&  =\P \tilde{\P }\left[  e^{-u}B\left(  \tau^{\left(  u\right)  }\right)
+\sqrt{1-e^{-2u}}\tilde{B}\left(  \tau^{\left(  u\right)  }\right)
>z~|~X=x\right]  .
\end{align*}
In the conditional probability above, since $x>-\mu$ and since $\tau=\arg
\min_{\left[  a,b\right]  }B$, we get $B\left(  \tau^{\left(  u\right)
}\right)  \geq B\left(  \tau\right)  =x+\mu+z$. Similarly, $\tilde{B}\left(
\tau^{\left(  u\right)  }\right)  \geq\tilde{B}\left(  \tilde{\tau}\right)  $
where $\tilde{\tau}:=\arg\min_{[a,b]}\tilde{B}$. In addition, we have that
$\tilde{B}\left(  \tilde{\tau}\right)  =\min_{[a,b]}\tilde{B}$ is independent
of $X$. Therefore we can write%
\begin{align*}
&  \mathrm{\tilde{E}E}\left[  \mathbf{1}_{B^{\left(  u\right)  }\left(
\tau^{\left(  u\right)  }\right)  >z}~|~X=x\right] \\
&  \geq\P \tilde{\P }\left[  e^{-u}\left(  x+\mu+z\right)  +\sqrt{1-e^{-2u}%
}\tilde{B}\left(  \tilde{\tau}\right)  >z~|~X=x\right] \\
&  =\tilde{\P }\left[  \min_{[a,b]}\tilde{B}>\frac{z\left(  1-e^{-u}\right)
-e^{-u}\left(  x+\mu\right)  }{\sqrt{1-e^{-2u}}}\right]  .
\end{align*}
Recall that we are trying to find a uniform lower bound on $g_{X}(x)$ for all
$x+\mu>0$; therefore, the term $-e^{-u}\left(  x+\mu\right)  $ in the
probability above will not help us even though it is negative, so we simply
ignore it, obtaining%
\[
\mathrm{\tilde{E}E}\left[  \mathbf{1}_{B^{\left(  u\right)  }\left(
\tau^{\left(  u\right)  }\right)  >z}~|~X=x\right]  \geq\tilde{\P }\left[
\min_{[a,b]}\tilde{B}>z\sqrt{\frac{1-e^{-u}}{1+e^{-u}}}\right]  =\tilde
{\P }\left[  \min_{[a,b]}\tilde{B}>z\sqrt{\tanh\left(  u/2\right)  }\right]
.
\]
\vspace*{0.1in}

\emph{Step 4: General tail lower bound for Gaussian infimum.} We are left to
find a lower bound on the last expression above. Since $z$ is arbitrary, this
is a general question about Gaussian infima, and can be solved using a
strategy similar to the one we are using for this entire proof of Proposition
\ref{sma}, albeit easier because we now have reduced the problem to studying
the tail of the random variable $\inf\tilde{B}$, which does not involved
positive parts, making the required Malliavin calculus computations more
straightforward. We state this result as a general Lemma of independent
interest, even though in the remainder of the proof of our Proposition
\ref{sma}, only the second statement of this Lemma, which is essentially
trivial, is needed.

\begin{lemma}
\label{Gaussinf} Let $B$ be a continuous scalar center Gaussian process on
$[a,b]$ satisfying \textnormal{(\ref{e1})} and \textnormal{(\ref{e1v})}. Assume $b-a$
is small enough to ensure that for all $s,t\in\lbrack a,b]$, $\mathrm{E}%
\left[  B\left(  s\right)  B\left(  t\right)  \right]  \geq\gamma^{2}\left(
a\right)  /2$. Define
\begin{align*}
\nu &  :=-\mathrm{E} \inf_{[a,b]}B=\left\vert \mathrm{E} \inf_{[a,b]}%
B\right\vert =\mathrm{E} \sup_{[a,b]}B,\\
\lambda &  :=\mathrm{E} \left[  \left\vert \inf_{[a,b]}B+\nu\right\vert
\right]  =\mathrm{E} \left[  \left\vert \sup_{[a,b]}B-\nu\right\vert \right]
.
\end{align*}
Then for any $y\geq\gamma\left(  b\right)  -\nu$,
\begin{align*}
&  \P \left[  \inf_{[a,b]}B>y\right] \\
&  =\P \left[  \sup B<-y\right]  \geq\lambda\frac{1-e^{-1}}{4}\frac{1}{y+\nu
}\exp\left(  -\frac{\left(  y+\nu\right)  ^{2}}{\gamma^{2}\left(  a\right)
}\right)  .
\end{align*}
Also note that for all $y\leq\gamma\left(  b\right)  $,
\begin{equation}
\P \left[  \inf_{[a,b]}B>y\right]  \geq\P \left[  \inf_{[a,b]}B>\gamma\left(
b\right)  \right]  =:\kappa>0. \label{kappadef}%
\end{equation}

\end{lemma}

Note that the positive constants $\kappa,\lambda,\nu,\gamma\left(  a\right)
,\gamma\left(  b\right)  $ depend only on $a,b$ and the law of $B$. The last
statement follows trivially from the fact that $\inf_{[a,b]}B$ has a positive
density on $\mathbb{R}$. That fact comes easily from Proposition \ref{dens}
and the fact that $\gamma^{2}\left(  a\right)  /2\leq G_{\inf_{[a,b]}B}%
\leq\gamma^{2}\left(  b\right)  $ almost surely, which is easy to prove using
the technique in Step 2. These inequalities are also useful to prove the first
statement of the lemma, via a modification of the proof of \cite[Corollary
4.5]{Viens:09}. The full proof of this lemma is left to the reader.\vspace
*{0.1in}

\emph{Step 5: Applying the lemma.}

We apply Lemma \ref{Gaussinf} to $\tilde{B}$, with $y=z\sqrt{\tanh\left(
u/2\right)  }$. Since $B$ and $\tilde{B}$ have the same law, all the
constants, particularly $\kappa$ in (\ref{kappadef}), are as in Lemma
\ref{Gaussinf}. Since $\tanh\leq1$ on $\mathbb{R}_{+}$, from the second part
of the lemma and the conclusion of Step 3 we get for every $z\leq\gamma\left(
b\right)  $, every $x>-\mu$, and every $u\geq0$,%
\begin{equation}
\mathrm{\tilde{E}E}\left[  \mathbf{1}_{B^{\left(  u\right)  }\left(
\tau^{\left(  u\right)  }\right)  >z}~|~X=x\right]  \geq\kappa.
\label{finalEE}%
\end{equation}
For $z\geq\gamma\left(  b\right)  $, we can still apply the second part of
Lemma \ref{Gaussinf} for $u$ small enough. Since $\tanh$ is bijective and
increasing on $\mathbb{R}_{+}$, we get $y=z\sqrt{\tanh\left(  u/2\right)
}\leq\gamma\left(  b\right)  $ if and only if $u\leq2\mathrm{arc}%
$\textrm{tanh}$\left(  \gamma^{2}\left(  b\right)  /z^{2}\right)  $. For such
a $u$ and $z$, and every $x>-\mu$, by the conclusion of Step 3, inequality
(\ref{finalEE}) still holds.\vspace*{0.1in}

\emph{Step 6. Conclusion. }

By the conclusion of Step 2, from (\ref{finalEE}), we can now write, for all
$z\leq\gamma\left(  b\right)  $ and all $x>-\mu$,%
\[
g_{X}\left(  x\right)  \geq\frac{\gamma^{2}\left(  a\right)  }{2}\kappa.
\]
By the conclusion of Step 2, we can find a lower bound on $g_{X}\left(
x\right)  $ by integrating only over the range $u\in\lbrack0,2\mathrm{arc}%
$\textrm{tanh}$\left(  \gamma^{2}\left(  b\right)  /z^{2}\right)  ]$, where
(\ref{finalEE}) is still valid: we get, for all $z\geq\gamma\left(  b\right)
$ and every $x>-\mu$,
\begin{align*}
g_{X}\left(  x\right)   &  \geq\frac{\gamma^{2}\left(  a\right)  }{2}%
\kappa\int_{0}^{2\mathrm{arctanh}\left(  \gamma^{2}\left(  b\right)
/z^{2}\right)  }du~e^{-u}\\
&  =\frac{\gamma^{2}\left(  a\right)  }{2}\kappa\left(
1-e^{-2\mathrm{arctanh}\left(  \gamma^{2}\left(  b\right)  /z^{2}\right)
}\right)  .
\end{align*}
With $F\left(  z\right)  $ as defined in the statement of Proposition
\ref{sma}, we summarize the two inequalities above as: for all $x>-\mu$ and
all $z\in\mathbb{R}$%
\[
g_{X}\left(  x\right)  \geq F\left(  z\right)  \frac{\gamma^{2}\left(
a\right)  \kappa}{2}.
\]

Thus by relations (\ref{densUB}) and (\ref{Deedee}), since either $z$ or $-z$
is positive,
\[
\P \left(  D_{1}\right)  \leq\varepsilon\frac{2\kappa\mathrm{E}\left[
\left\vert X\right\vert \right]  }{\gamma^{2}\left(  a\right)  }\left(
1+\frac{1}{F\left(  \left\vert z\right\vert \right)  }\right)  .
\]
Finally, we can easily estimate $\mathrm{E}\left[  \left\vert X\right\vert
\right]  $, and show that this does not depend on $z$. Indeed, from
(\ref{GXcomput})\ and Cauchy-Schwartz, we immediately get $G_{X}\leq$
$\gamma^{2}\left(  b\right)  $. Then, from (\ref{NP}), with $f=$ the identity,
we get $\mathrm{E}\left[  X^{2}\right]  =\mathrm{E}\left[  G_{X}\right]
\leq\gamma^{2}\left(  b\right)  $. Hence by Jensen, $\mathrm{E}\left[
\left\vert X\right\vert \right]  \leq\gamma\left(  b\right)  $. Plugging this
into the last estimate of $\P \left(  D_{1}\right)  $, and combining this with
the estimate of $\P \left(  D_{2}\right)  $ from Theorem \ref{hitd1}, finishes
the proof of the proposition's first statement. The second statement follows
immediately from Corollary \ref{hitd1cor}.
\end{proof}

Using a covering argument and Proposition \ref{sma} one obtains the following
upper bound for the hitting probabilities of $B$ in terms of Hausdorff measure
(see \cite[Theorem 3.1]{Dalang:07} where a similar argument is performed).

\begin{thm}
\label{t3} Assume that the function $\varphi(s)=s^{d}/\gamma^{-1}(s)$ is
right-continuous and non-decreasing near $0$ with $\lim_{0+}\varphi=0$. Also
assume that $\gamma$ satisfies the condition \textnormal{(\ref{gammamult})} from
Corollary \ref{hitd1cor}. Then for all $0<a<b<\infty$, any $M>0$, there exists
a constant $C>0$ depending only on $a,b$, the law of $B$, and $M$, such that
for any Borel set $A\subset\lbrack-M,M]^{d}$,%
\[
\P ({B}([a,b])\cap A\neq\varnothing)\leq C{\mathcal{H}}_{\varphi}(A).
\]

\end{thm}

\begin{proof}
For all positive integers $n$, consider the intervals of the form
\[
I_{j}^{n}:=[t_{j}^{n},t_{j+1}^{n}],\qquad\text{ where }\qquad t_{j}%
^{n}:=j\gamma^{-1}(2^{-n}).
\]
Fix $\varepsilon\in\,(0\,,1)$ and $n\in\mathbb{N}$ such that $2^{-n-1}%
<\varepsilon\leq2^{-n}$, and write, for any $z\in A$,%
\[
\P \left(  B([a,b])\cap B(z,\varepsilon)\neq\varnothing\right)  \leq
\sum_{j:I_{j}^{n}\cap\lbrack a,b]\neq\varnothing}\P (B^{\gamma}(I_{j}^{n})\cap
B(z\,,\varepsilon)\neq\varnothing).
\]
The number of $j$'s involved in the sum is at most $(b-a)/\gamma^{-1}(2^{-n}%
)$. Also note that the diameter $\eta$ of $I_{j}^{n}$ is $\gamma^{-1}\left(
2^{-n}\right)  $, and therefore, $\gamma\left(  \eta\right)  =2^{-n}%
<2\varepsilon$. Then, for $\varepsilon\leq\varepsilon_{0}$ small enough, we
can apply Proposition \ref{sma} to each interval $I_{j}^{n}$. The constant
$\kappa$ in this proposition must then be replaced by $\kappa_{j}:=\P \left[
\inf_{I_{j}^{n}}B>\gamma\left(  t_{j}^{n}\right)  \right]  $. In any case, we
may use the uniform bound $\kappa_{j}\leq1$. Hence, Proposition \ref{sma}
implies that for all large $n$ and $z\in\mathbb{R}^{d}$,%
\begin{equation}%
\begin{split}
&  \P \left(  {B}\left(  [a,b]\right)  \cap B(z\,,\varepsilon)\neq
\varnothing\right) \\
&  \leq\frac{(b-a)}{\gamma^{-1}\left(  2^{-n}\right)  }\cdot\left(
\varepsilon\frac{2\gamma\left(  b\right)  }{\gamma^{2}\left(  a\right)
}\left(  1+\frac{1}{F\left(  M\right)  }\right)  +\frac{L\sqrt{\ell}}%
{\gamma\left(  a\right)  }2\varepsilon\right)  ^{d}\\
&  \leq\frac{\varepsilon^{d}}{\gamma^{-1}\left(  \varepsilon\right)  }\left(
b-a\right)  \left(  \frac{2\gamma\left(  b\right)  }{\gamma^{2}\left(
a\right)  }\left(  1+\frac{1}{F\left(  M\right)  }\right)  +\frac{2L\sqrt
{\ell}}{\gamma\left(  a\right)  }\right)  ^{d}\\
&  =:C\varphi(\varepsilon).
\end{split}
\label{eq:hit:ball:UB}%
\end{equation}
In the first inequality we used the fact that the endpoints of each interval
$I_{j}^{n}$ are bounded above by $b$ and below by $a$, and we appealed to the
fact that $\gamma$ is increasing; in the last inequality we used again the
fact that $\gamma$ is increasing, and $\varepsilon\leq2^{-n}$. Observe that
\[
C:=\left(  b-a\right)  \left(  \frac{2\gamma\left(  b\right)  }{\gamma
^{2}\left(  a\right)  }\left(  1+\frac{1}{F\left(  M\right)  }\right)
+\frac{2L\sqrt{\ell}}{\gamma\left(  a\right)  }\right)  ^{d}%
\]
does not depend on $n\,,\varepsilon$, or $A$, except via the value $M$.
Therefore, (\ref{eq:hit:ball:UB}) is valid for all $\varepsilon\in
\,(0\,,\varepsilon_{0})$.

Now we use a covering argument: Choose $\varepsilon\in\,(0,\varepsilon_{0})$
and let $\{B(z_{i},r_{i})\}_{i=1}^{\infty}$ be a sequence of open balls in
$\mathbb{R}^{d}$ with radii $r_{i}\in\,(0\,,\varepsilon]$ such that
\begin{equation}
A\subseteq\bigcup_{i=1}^{\infty}B(z_{i},r_{i})\quad\text{and}\quad\sum
_{i=1}^{\infty}\varphi(2r_{i})\leq\mathcal{H}_{\varphi}(A)+\varepsilon,
\label{eq:HHHH}%
\end{equation}
where $\varphi(r)=r^{d}/\gamma^{-1}(r)$. Then, \eqref{eq:hit:ball:UB} and
\eqref{eq:HHHH} together imply that
\[%
\begin{split}
\P \left(  {B}([a,b])\cap A\neq\varnothing\right)   &  \leq\sum_{i=1}^{\infty
}\P ({B}([a,b])\cap B(z_{i},r_{i})\neq\varnothing)\\
&  \leq C\sum_{i=1}^{\infty}\varphi(2r_{i})\\
&  \leq C(\mathcal{H}_{\varphi}(A)+\varepsilon).
\end{split}
\]
Finally, let $\varepsilon\rightarrow0^{+}$ to deduce the desired upper bound.
\end{proof}

In Section 5 we will see what our main theorems \ref{tcap} and \ref{t3} mean
in a specific two-parameter class of examples with highly non-stationary
increments. We finish this section with a general discussion of how close our
canonical kernel functions $\mathrm{K}$ and $1/\varphi$ are to each other, as
identified in the Hausdorff measure upper bound (Theorem \ref{t3}) and the
capacity lower bound (Theorem \ref{tcap}). To fix ideas, recall that in the
case of fBm, the relevant function for the Hausdorff measure is $\varphi
\left(  r\right)  =r^{d-1/H}$, and that the Hausdorff measure result applies
when $d>1/H$. In that same case, the capacity lower bound uses the Newtonian
kernel $\mathrm{K}=\mathrm{K}_{d-1/H}$, meaning that $\mathrm{K}=1/\varphi$,
or at least, since all results are given modulo multiplicative constants
(depending on $a,b$ and the law of $B$), and the values obtained in the bounds
depend qualitatively only on the behavior of $\mathrm{K}$ and $1/\varphi$ near
$0$, the fBm case shows that what is important is that lower and upper bounds
refer to commensurate canonical functions $\mathrm{K}$ and $1/\varphi$ near
$0$.

In the general case, we would like to know to what extent we still have that
$\mathrm{K}$ and $1/\varphi$ are commensurate near $0$. Recall that%
\[
\mathrm{K}\left(  x\right)  =\max\left(  1,v\left(  \gamma^{-1}\left(
x\right)  \right)  \right)  ,
\]
where $v\left(  r\right)  =\int_{r}^{b-a}ds/\gamma^{d}\left(  s\right)  $, and%
\[
1/\varphi\left(  x\right)  =\gamma^{-1}\left(  x\right)  /x^{d}.
\]
Thus, their commensurability near $0$ is equivalent to that of $v\left(
r\right)  $ and $r\mapsto r/\gamma^{d}\left(  r\right)  $ for $r$ near $0$.
Since all these functions are continuous and non-zero everywhere except at
$r=0$, we only need to investigate whether%
\[
0<\liminf_{r\rightarrow0}\frac{r/\gamma^{d}\left(  r\right)  }{v\left(
r\right)  }\leq\limsup_{r\rightarrow0}\frac{r/\gamma^{d}\left(  r\right)
}{v\left(  r\right)  }<+\infty.
\]
This comparison is only fair if one also takes into account the assumptions
used in Theorems \ref{t3} and \ref{tcap}, which include the concavity of
$\gamma$, and the fact that $\lim_{0}\varphi=0$ and $\varphi$ is
non-decreasing. To make this presentation more elementary, we specialize to
the case where $\gamma$ is differentiable, but similar arguments can be
developed in the case where differentiability holds only almost everywhere.

Since we now assume that $\lim_{0}\varphi=0$ and $\varphi$ is non-decreasing,
there exists an increasing sequence of constant $c_{n}$ with $\lim_{n}%
c_{n}=+\infty$, and a decreasing sequence of constants $r_{n}\in(0,b-a]$ in
with $\lim_{n}r_{n}=0$ such that for every $r\in\lbrack r_{n+1},r_{n}]$,
$1/\gamma^{d}\left(  r\right)  \geq c_{n}/r$. Therefore, for any integer $N$,
\begin{align*}
v\left(  r_{N}\right)   &  =\int_{r_{N}}^{b-a}\frac{ds}{\gamma^{d}\left(
s\right)  }\geq\sum_{n=1}^{N-1} c_{n}\left(  \ln r_{n}-\ln r_{n+1}\right) \\
&  \geq c_{1} \left(  \ln\left(  \frac{1}{r_{N}}\right)  -\ln\left(  \frac
{1}{b-a}\right)  \right)  .
\end{align*}
Since this expression goes to $+\infty$ with $N$, $v$ is unbounded. We may
thus apply l'H\^{o}pital's rule similarly to what we did in Step 4 of the
proof of Theorem \ref{tcap}, to get that
\[
\lim_{r\rightarrow0}\frac{r/\gamma^{d}\left(  r\right)  }{v\left(  r\right)
}=-1+d\lim_{r\rightarrow0}\frac{r\gamma^{\prime}\left(  r\right)  }%
{\gamma\left(  r\right)  }%
\]
if the last limit above exists. We already know by concavity of $\gamma$ that
the last expression above is bounded above by $d-1$ (see Step 4 of the proof
of Theorem \ref{tcap}). Therefore, by requiring that it be bounded away from
$0$, we can assert the following.

\begin{prop}
\label{commens}Under the assumptions in Theorems \ref{t3} and \ref{tcap}, if
$\lim_{r\rightarrow0}\frac{r\gamma^{\prime}\left(  r\right)  }{\gamma\left(
r\right)  }$ exists, then the functions $\mathrm{K}$ and $1/\varphi$ are
commensurate if and only if%
\[
d>1/\lim_{r\rightarrow0}\frac{r\gamma^{\prime}\left(  r\right)  }%
{\gamma\left(  r\right)  }.
\]

\end{prop}

The advantage of this criterion is that it separates the dimension $d$ from
the information contained in $\gamma$ about the law of the scalar process $B$.
One can also reformulate the above proposition by assuming that $\gamma$ is of
index $\mathrm{ind}\gamma$ (see the discussion surrounding Lemma
\ref{Lemindex} for the significance of the index): $d>1/\mathrm{ind}\gamma$
implies that $\mathrm{K}$ and $1/\varphi$ are commensurate; the proof of this
fact is left to the reader.

\section{Examples}

In this section, we look at a class of examples, to see what our results of
Sections 2 and 4 imply in practice. Before we do this, let us establish and
recall what these results imply in general for the probabilities of hitting
points (singletons).

\begin{thm}
\label{t4}Assume that $\gamma$ satisfies Hypothesis \ref{h0} and that $B$ is
such that (\ref{e1}) and (\ref{e1v}) hold. If $1/\gamma^{d}$ is integrable at
$0$, then $B$ hits points with positive probability.

Assume instead that $\gamma\left(  r\right)  =o\left(  r^{1/d}\right)  $ near
$0$, and that $\varphi(s):=s^{d}/\gamma^{-1}(s)$ is non-decreasing near $0$
and Condition \textnormal{(\ref{gammamult})} holds. Then almost surely, $B$ does
not hit points.
\end{thm}

\begin{proof}
The first statement of the theorem was already established in Remark
\ref{hitprem}. To prove the second statement of the theorem, first note that
since $\gamma$ is continuous and strictly increasing, $\varphi$ is a
continuous function. The assumption of the theorem also says that $\varphi$ is
non-decreasing. We claim that $\lim_{0}\varphi=0$. Assuming this is true, we
can apply Theorem \ref{t3}. Thus,
\begin{align*}
{\mathcal{H}}_{\varphi}(\left\{  x\right\}  )  &  =\lim_{\varepsilon
\rightarrow0^{+}}\inf\left\{  \sum_{i=1}^{\infty}\varphi(2r_{i}):x\in
\bigcup_{i=1}^{\infty}B(x_{i},r_{i}),\ \sup_{i\geq1}r_{i}\leq\varepsilon
\right\} \\
&  =\lim_{\varepsilon\rightarrow0^{+}}\inf\left\{  \varphi(2\varepsilon):x\in
B(x,\varepsilon)\right\}  =\lim_{\varepsilon\rightarrow0^{+}}\inf
\varphi(2\varepsilon)=0,
\end{align*}
finishing the proof.

We are left to prove that $\lim_{0}\varphi=0$. Since $\gamma$ is bijective, to
compute its inverse, we can solve for $r$ in $\gamma\left(  r\right)  =s$, and
we get $\lim_{0}\gamma^{-1}=0$. Thus in the relation $\gamma\left(  r\right)
=o\left(  r^{1/d}\right)  $ we can replace $r$ by $\gamma^{-1}\left(
s\right)  $ to get, for $s$ near $0$, $\gamma\left(  \gamma^{-1}\left(
s\right)  \right)  =o\left(  \left(  \gamma^{-1}\left(  s\right)  \right)
^{1/d}\right)  $, which, after taking the power of $d$ on both sides, yields
$s^{d}=o\left(  \gamma^{-1}\left(  s\right)  \right)  $. Dividing by
$\gamma^{-1}\left(  s\right)  $ yields the result.
\end{proof}

\begin{rmk}
We have already discussed, in Remark \ref{GermanRem}, that a classical method
based on the existence of jointly continuous local time provides a more
restrictive sufficient condition for hitting points than our Theorem \ref{t4},
which only requires that $1/\gamma^{d}$ be integrable at $0$.
\end{rmk}

\begin{rmk}
\label{modofcontmethod}There is also a classical strategy for proving that a
Gaussian process does not hit points, based on its modulus of continuity. The
method was presented in \cite{Khoshnevisan:97}. We can apply this method in
our context. It is known (see \cite{Adler:90}) that under condition
(\ref{e1}), $h(r)=\gamma(r)\log^{1/2}\left(  \frac{1}{r}\right)  $ is a
uniform modulus of continuity of $B$. The method of \cite{Khoshnevisan:97} can
be used to prove that if $h^{d}\left(  \varepsilon\right)  =o\left(
\varepsilon\right)  $, then $B$ hits points with probability zero; all details
are omitted. We will see below in Remark \ref{modofcontcompare} that this
condition is more restrictive than the one we give here in the second part of
Theorem \ref{t4}.\bigskip
\end{rmk}

We define a $2$-parameter collection of Gaussian processes as follows: for
every $\beta\in\mathbb{R},H\in(0,1)$, we will use the notation $B^{H,\beta}$
for any process $B$ satisfying (\ref{e1}) and (\ref{e1v}) with, for every $r$
in a closed interval in $[0,1)$,%
\begin{equation}
\gamma(r)=\gamma_{H,\beta}\left(  r\right)  :=r^{H}\log^{\beta}(\frac{1}{r}).
\label{gHb}%
\end{equation}
It should be noted that since the constants in (\ref{e1}) are not equal to
$1$, there is a considerable amount of flexibility in how each $B^{H,\beta}$
is defined, that is to say, for $H$ and $\beta$ fixed, $B^{H,\beta}$
represents a generic element of an entire family, constrained only by
(\ref{e1}) and (\ref{e1v}). Thus the notation $B^{H,\beta}$ can be understood
as a class of processes, or a representative member of this class.

To define $B$ on a larger time interval than $[0,1)$, one may replace
$\log^{\beta}(\frac{1}{r})$ by $\log^{\beta}(\frac{c}{r})$ for some
appropriately small constant $c$, but we will not consider this extension. Nor
will we consider the case where the formula for $\gamma_{H,\beta}$ has a
leading constant $c$, nor the case where $\gamma$ is only assumed to be
commensurate with $\gamma_{H,\beta}$ defined in (\ref{gHb}). All these
additional cases can be treated just as we do below using either trivial or
straightforward extensions, with essentially identical results. We omit any
further discussion of these cases for the sake of conciseness.

When $\beta=0$, the family covers fractional Brownian motion, and is
essentially the class studied in \cite{Bierme:09}. Those processes are not
self-similar, but have the same behavior as fBm in terms of their hitting
probabilities. Indeed, one easily checks that the results of our Theorems
\ref{t3} and \ref{tcap} translate into the same results as for the fBm with
the corresponding $H$: the capacity lower bound holds with potential kernel
$\mathrm{K}=\mathrm{K}_{d-1/H}$ for any $d$, and the Hausdorff measure upper
bound holds with the function $\varphi\left(  r\right)  =r^{d-1/H}$ when
$d>1/H$.

When $\beta\neq0$, the processes in this family are highly non-self-similar.
In particular, for $H$ fixed, if $\beta>0$, $B^{H,\beta}$ is infinitely more
irregular than the fBm $B^{H}$, and if $\beta<0$ it is infinitely more regular
than $B^{H}$. By the classical Dudley-Fernique-type results on regularity of
Gaussian fields (see \cite{Adler:90}), it is easy to check that $r\mapsto
r^{H}\log^{\beta+1/2}\left(  1/r\right)  $ is an almost-sure modulus of
continuity for $B^{H,\beta}$. Thus the three processes $B^{H}$, $B^{H,\beta}$
for $\beta<0$, $B^{H,\beta}$ for $\beta>0$, share the property that they are
$\alpha$-H\"{o}lder-continuous almost surely as soon as $\alpha<H$. As a
matter of fact, if $\beta<-1/2$, $B^{H,\beta}$ is almost surely $H$-H\"{o}lder
continuous, but not $\alpha$-H\"{o}lder-continuous almost surely if $\alpha>H$.

We now explore what the theorems in this article imply for $B^{H,\beta}$, and
will see that for the most part, the potential kernel $\mathrm{K}%
=\mathrm{K}_{d-1/H}$ and the Hausdorff measure function $\varphi\left(
r\right)  =r^{d-1/H}$ need to be abandoned. We will also see that the case
$H=1/d$ represents a critical situation, in which a transition occurs on the
question of hitting points, depending on the value of $\beta$.\vspace*{0.1in}

Let us first translate Theorem \ref{t4} on probabilities of hitting points for
$B^{H,\beta}$.

\begin{prop}
\label{hitp} If $d<1/H$, or if $d=1/H$ and $\beta>1/d$, then any process
$B^{H,\beta}$ hits points with positive probability. On the other hand, if
$d>1/H$ or if $d=1/H$ and $\beta<0$, any process $B^{H,\beta}$ a.s. does not
hit points.
\end{prop}

\begin{proof}
The first statement of the proposition follows by the first statement in
Theorem \ref{t4}, since
\[
\int_{0}\frac{ds}{\gamma^{d}\left(  s\right)  }=\int_{0}\frac{ds}{s^{dH}%
\log^{d\beta}\left(  1/s\right)  }<\infty
\]
holds as soon as $dH<1$ or as soon as $dH=1$ and $d\beta>1$ modulo checking
Hypothesis \ref{h0}, which we now do. Since $\gamma^{\prime}\left(  r\right)
$ exists and is equal to $r^{H-1}\left(  H\log^{\beta}\left(  1/r\right)
-\beta\log^{\beta-1}\left(  1/r\right)  \right)  $, we see that for small $r$,
this is strictly positive, asymptotically equivalent to $r^{H-1}H\log^{\beta
}\left(  1/r\right)  $, and strictly decreasing. Therefore Hypothesis \ref{h0} holds.

The proof of the second statement of the proposition follows from the second
statement of Theorem \ref{t4} in an equally straightforward way, whose details
are omitted.
\end{proof}

\begin{rmk}
\label{modofcontcompare}By the second statement of this proposition,
$B^{H,\beta}$ hits points with probability zero when $H=1/d$ as soon as
$\beta<0$. If we try to use the Gaussian modulus-of-continuity method
described in Remark \ref{modofcontmethod}, to get the same result when
$H=1/d$, we see that we must require that $h\left(  r\right)  ^{d}r^{-1}%
=\log^{d\beta+d/2}\left(  1/r\right)  $ tends to $0$ as $r\rightarrow0$, i.e.
that $\beta<-1/2$. Thus the second part of Theorem \ref{t4} is sharper than
the method described in Remark \ref{modofcontmethod}.
\end{rmk}

We next look at what Theorems \ref{t3} and \ref{tcap} imply on bounds for the
hitting probabilities of $B^{H,\beta}$ for arbitrary sets.

\begin{thm}
\label{texa}Assume $B=B^{H,\beta}$, i.e. assume that for each component of
$B$, \textnormal{(\ref{e1})} and \textnormal{(\ref{e1v})} hold with $\gamma
=\gamma_{H,\beta}$ as in \textnormal{(\ref{gHb})}. Then the following statements hold.

\begin{enumerate}
\item If $d>1/H$, for all $0<a<b<1$ and $M>0$, there exist constants
$C_{1},C_{2}>0$ such that for any Borel set $A\subset[-M, M]^{d}$,
\[
C_{1}\mathcal{C}_{1/\varphi}(A)\leq\P ({B}([a,b])\cap A\neq\varnothing)\leq
C_{2}{\mathcal{H}}_{\varphi}(A),
\]
where $\varphi(x)=x^{d-\frac{1}{H}}\log^{\beta/H}(1/x).$

\item If $d=1/H$ and $\beta<0$, the upper bound still holds, with the same
$\varphi$, namely $\varphi(x)=\log^{\beta/H}(1/x).$

\item If $d=1/H$, for $\beta<1/d$, the lower bound holds with $\varphi
(x)=\log^{\beta/H-1}\left(  1/x \right)  $.

\item If $d=1/H$, for $\beta\geq1/d$, the lower bound holds with
$\varphi\equiv1$.

\item If $d<1/H<+\infty$ the lower bound holds with $\varphi\equiv1$.
\end{enumerate}
\end{thm}

\begin{rmk}
Notice that in the case $d=1/H$, for $\beta<0$, there is a discrepancy factor
equal to $\log(1/x)$ between the two functions $\varphi$ in the upper and
lower bounds. This lack of precision at the logarithmic level is not visible
in the power scales. It shows that in the so-called \textquotedblleft critical
case\textquotedblright\ identified for fBm and other power-scale-based
processes as in \cite{Bierme:09}, at least one of the lower capacity bounds or
the upper Hausdorff measure bounds must be inefficient.
\end{rmk}

The theorem and its corollaries given below are all proved further below.
Unlike in the power scale, our theorems allow us to look at our examples when
$H=0$ or $1$. When $H=1$, we get non-trivial (non-smooth) processes as soon as
$\beta>0$. When $H=0$, we get continuous processes as soon as $\beta<-1/2$; in
this case, Condition (\ref{gammamult}) does not hold, so care is required for
the upper bound.

\begin{cor}
\label{corex1}Assume $B,a,b,M,A$ are as in Theorem \ref{texa}. If $H=1$ and
$\beta>0$, for all $d>1$, both bounds in Theorem \ref{texa} hold with
$\varphi(x)=x^{d-1}\log^{\beta}(1/x)$. If $H=0$ and $\beta<-1/2$, then the
lower bound in Theorem \ref{texa} hold with $\varphi=1$, so that $B$ hits
points with positive probability.
\end{cor}

We can also construct uncountable Borel sets which are polar for one process
and are visited with positive probability for another, with both processes
having the same H\"{o}lder continuity properties. In the next corollary, we
consider the chance for a process in $\mathbb{R}^{d}$ to hit a linear Cantor
set (a subset of the $x$-axis). While our method applies to a variety of
Cantor sets (see for instance the $p$-Cantor sets in \cite{Cabrelli:04}), for
simplicity, we consider first the classical Cantor set with fixed ratio
$q\in\left(  0,1/2\right)  $ defined as follows. Let $A_{0}=[0,1]$; for
$n\in\mathbb{N}$, assuming $A_{n}$ has been defined and is a union of $2^{n}$
intervals of length $q^{n}$, we define $A_{n+1}$ by removing a central open
interval of length $q^{n}\left(  1-2q\right)  $ from each interval in $A_{n}$.
The Cantor set is $A:=\lim_{n\rightarrow\infty}A_{n}=\cap_{n}A_{n}$. It is
known that $A$ has Hausdorff dimension $d\left(  A\right)  =\ln2/\ln\left(
1/q\right)  \in\left(  0,1\right)  $. It is also known that it has positive
$\mathrm{K}$-capacity $\mathcal{C}_{K}\left(  A\right)  >0$ if and only if
$\sum_{n=1}^{\infty}2^{-n}\mathrm{K}\left(  q^{n}\right)  <\infty$. See
\cite{Beardon:68}. Moreover, it is easy to check that the $d\left(  A\right)
$-Hausdorff measure of $A$ satisfies $\mathcal{H}_{d\left(  A\right)  }\left(
A\right)  \leq1$; indeed $A_{n}=:\sum_{j=1}^{2^{n}}A_{n,j}$ is a covering of
$A$ with intervals $A_{n,j}$ of length $q^{n}$, which can be made arbitrarily
small, and $\sum_{j=1}^{2^{n}}\left\vert A_{n,j}\right\vert ^{d\left(
A\right)  }=\sum_{j=1}^{2^{n}}q^{d\left(  A\right)  n}=\left(  2q^{d\left(
A\right)  }\right)  ^{n}=1$ since $q^{d\left(  A\right)  }=2^{-1}$. With all
these facts, we state the following.

\begin{cor}
\label{corex2}For any dimension $d\geq2$, let $H\in\left(  1/d,1/\left(
d-1\right)  \right)  $. Assume $A$ is a binary Cantor set on the $x$-axis of
$\mathbb{R}^{d}$ with constant ratio $q:=2^{-1/\left(  d-1/H\right)  }$, so
that its Hausdorff dimension is $d-1/H\in\left(  0,1\right)  $. Then $A$ is
polar for any process in the class $B^{H,\beta}$ with $\beta<0$, i.e. with
probability 1, $B^{H,\beta}$ does not hit $A$ during the time interval
$[a,b]\subset(0,\infty)$. On the other hand, for $\beta^{\prime}>H$, any
process in the class $B^{H,\beta^{\prime}}$ hits $A$ with positive probability
during the time interval $[a,b]$.

Note that processes in the classes $\left\{  B^{H,\beta}:-1/2\leq
\beta<0\right\}  $ and $\left\{  B^{H,\beta^{\prime}}:\beta^{\prime}\geq
H\right\}  $ share the same H\"{o}lder continuity in the sense that they are
$\alpha$-H\"{o}lder-continuous a.s. if and only if $\alpha<H$.
\end{cor}

The previous corollary shows that our results help us construct processes that
hit classical Cantor sets, and others that do not even though their path
regularity is very similar to the first ones. This is an improvement over
H\"{o}lder-scale tools such as those in \cite{Bierme:09}: Corollary
\ref{corex2} cannot be established with those tools, since it is known that a
Cantor set with constant ratio $q=2^{-1/\left(  d-1/H\right)  }$ has null
$\left(  d-1/H\right)  $-capacity, so that capacity lower bounds for hitting
probabilities are inconclusive; similarly, a positive $\left(  d-1/H\right)
$-Hausdorff measure does not help prove whether a set is non-polar.

On the other hand, even our results leave a small gap in the analysis: the
case $\beta\in\lbrack0,H]$ is not covered. For instance, our results are not
fine enough to tell whether classical fBm (a member of the class $B^{H,0}$,
i.e. with $\beta=0$) in $\mathbb{R}^{d}$ with Hurst parameter $H$ hits a
classical linear Cantor set with dimension $d-1/H$ on the $x$-axis. However,
our results are fine enough to show us how to construct a generalized Cantor
set with the same dimension $d-1/H$, which fBm does hit. For such a
construction, referring again to \cite{Beardon:68}, consider a sequence
$\left(  q_{n}\right)  _{n\in\mathbb{N}}$ of scalars in $(0,1/2)$, and
construct a Cantor set $\tilde{A}=\cap_{n}\tilde{A}_{n}$ like we did above for
$A$, except that from levels $n$ to $n+1$, we remove the central open interval
of length $\left(  \prod_{i=1}^{n}q_{i}\right)  \left(  1-2q_{n+1}\right)  $.
Then it is known that for a given capacity kernel \textrm{K}, the
\textrm{K}-capacity $\mathcal{C}_{\mathrm{K}}( \tilde{A}) $ is strictly
positive if and only if $\sum_{n}2^{-n}$\textrm{K}$\left(  \prod_{i=1}%
^{n}q_{i}\right)  $ is a convergent series. We now devise a sequence $\left(
q_{n}\right)  _{n}$ such that the resulting Cantor $\tilde{A}$ set is slightly
bigger than the classical set for which $q_{n}\equiv q=2^{-1/\left(
d-1/H\right)  }$, just big enough for us to guarantee a positive $\left(
d-1/H\right)  $-capacity, and therefore for any process in the class $B^{H,0}%
$, including classical fBm, to hit $\tilde{A}$ with positive probability.

\begin{cor}
\label{corex3}Let $c>1$. For any dimension $d\geq2$, let $H\in\left(
1/d,1/\left(  d-1\right)  \right)  $. Assume $\tilde{A}$ is a generalized
Cantor set on the $x$-axis of $\mathbb{R}^{d}$ with level-$n$ ratio $q_{n}%
\geq\left(  2^{-1}\left(  1-c/n\right)  \right)  ^{1/\left(  d-1/H\right)  }$,
constructed as described above. Then any process in the class $B^{H,0}$,
including fBm, hits $\tilde{A}$ with positive probability during the time
interval $[a,b]\subset(0,\infty)$.

If the inequality on $q_{n}$ is replaced by an equality, then the Hausdorff
dimension of $\tilde{A}$ is $d-1/H\in\left(  0,1\right)  $.
\end{cor}

\begin{proof}
[Proof of Theorem \ref{texa}]\emph{Step 0: checking the assumptions of the
theorems.}

Throughout this entire proof, we will use the fact that the conclusions of
Theorems \ref{t3} and \ref{tcap} can be stated for any functions that are
commensurate with the $\varphi$ and $\mathrm{K}$ defined therein.

In the proof of Proposition \ref{hitp}, we already checked that $\gamma
_{H,\beta}$ satisfies Hypothesis \ref{h0}, so that we may apply Theorem
\ref{tcap}. To be allowed to apply Theorem \ref{t3}, which we only need for
$d>1/H$ or $\left\{  d=1/H;\beta<0\right\}  $, we only need to check the following:

\begin{description}
\item[(i)] $\varphi(x)=x^{d}/\gamma^{-1}(x)$ is right-continuous and
non-decreasing near $0$ with $\lim_{0+}\varphi=0$.

\item[(ii)] $\gamma$ satisfies the condition (\ref{gammamult}).
\end{description}

Condition (ii) is easy for $\beta\geq0$ since for $x,y<1/e$, we have
$\log\left(  1/x\right)  ,\log\left(  1/y\right)  >1$ and thus
\begin{align*}
\gamma\left(  x\right)  \gamma\left(  y\right)   &  =\left(  xy\right)
^{H}\log^{\beta}\left(  1/x\right)  \log^{\beta}\left(  1/y\right) \\
&  \geq\left(  xy\right)  ^{H}2^{-\beta}\left(  \log\left(  1/x\right)
+\log\left(  1/y\right)  \right)  ^{\beta}=2^{-\beta}\gamma\left(  xy\right)
,
\end{align*}
so that we get $\int_{0}^{1/2}\gamma\left(  xy\right)  \frac{dy}{y\sqrt
{\log\left(  1/y\right)  }}\leq2^{\beta}\gamma\left(  x\right)  \int_{0}%
^{1/2}\frac{dy}{y^{1-H}}=\gamma\left(  x\right)  2^{H+\beta}/H$. When
$\beta<0$, on the other hand, since for $x<1$, $\log\left(  1/\left(
xy\right)  \right)  \geq\log\left(  1/x\right)  $, we get%
\begin{align*}
\int_{0}^{1/2}\gamma\left(  xy\right)  \frac{dy}{y\sqrt{\log\left(
1/y\right)  }}  &  \leq x^{H}\int_{0}^{1/2}y^{H}\log^{-\beta}\left(
1/x\right)  \frac{dy}{y\sqrt{\log\left(  1/y\right)  }}\\
&  =\gamma\left(  x\right)  \int_{0}^{1/2}\frac{dy}{y^{1-H}\sqrt{\log\left(
1/y\right)  }}\leq\gamma\left(  x\right)  2^{H}/2.
\end{align*}

To prove condition (i), we start by noting that for small $x$, $\gamma
^{-1}\left(  x\right)  $ is commensurate with $x^{1/H}\log^{-\beta/H}\left(
1/x\right)  $. Since we may replace $\varphi$ by a constant multiple of it to
apply to Theorem \ref{t3}, we only need to check that $x\mapsto x^{d-1/H}%
\log^{\beta/H}\left(  1/x\right)  $ is right-continuous and non-decreasing
near $0$. Since we assume that $d>1/H$ or $\left\{  d=1/H;\beta<0\right\}  $,
this is immediately true.\vspace*{0.1in}

\emph{Step 1: Proof of case 1.} By Theorem \ref{t3}, the upper bound in case 1
holds with $\varphi\left(  x\right)  =x^{d}/\gamma^{-1}\left(  x\right)  $.
Since $\gamma^{-1}\left(  x\right)  $ is commensurate with $x^{1/H}%
\log^{-\beta/H}\left(  1/x\right)  $, the upper bound conclusion of case 1
holds. By Theorem \ref{tcap}, the lower bound in case 1 holds with the
capacity $\mathcal{C}_{\mathrm{K}}$ where $\mathrm{K}=v\circ\gamma^{-1}$. By
Proposition \ref{commens}, the lower bound conclusion of case 1 will hold as
soon as we can show that $d>1/\lim_{r\rightarrow0}r\gamma^{\prime}\left(
r\right)  /\gamma\left(  r\right)  $. We have $\gamma^{\prime}\left(
r\right)  =r^{H-1}\left(  H\log^{\beta}\left(  1/r\right)  -\beta\log
^{\beta-1}\left(  1/r\right)  \right)  $ so that the limit above computes as
$1/H$, finishing the proof of case 1.\vspace*{0.1in}

\emph{Step 2: Proof of case 2.} This case follows using the argument and
computation in Step 1, since the original $\varphi\left(  x\right)  $ in
Theorem \ref{t3} is commensurate with $\log^{-\beta/H}\left(  1/x\right)
$.\vspace*{0.1in}

\emph{Step 3: Proof of case 3 and 4. }If $d=1/H$, for any $\beta$, we may use
Theorem \ref{tcap} with $\mathrm{K}=\max\left(  1,v\circ\gamma^{-1}\right)  $,
but one may check that Proposition \ref{commens} is not applicable, so we need
to compute $v$ directly. Since everything only needs to be determined up to
multiplicative constants, we do not keep track of them. We compute
\begin{align*}
v\left(  r\right)   &  =\int_{r}^{b-a}s^{-1}\log^{-\beta/H}\left(  1/s\right)
ds\\
&  =\frac{1}{1-\beta H}\left(  \log^{1-\beta/H}\left(  1/r\right)
-\log^{1-\beta/H}\left(  1/b-a\right)  \right) \\
&  \asymp\log^{1-\beta/H}\left(  1/r\right)  .
\end{align*}
Since, as we said above, $\gamma^{-1}\left(  x\right)  \asymp x^{1/H}%
\log^{-\beta/H}\left(  1/r\right)  $, it is then easy to get%
\begin{align*}
v\circ\gamma^{-1}\left(  x\right)   &  \mathrm{\asymp}\log^{1-\beta/H}\left(
x^{-1/H}\log^{\beta/H}\left(  1/r\right)  \right) \\
&  \mathrm{\asymp}\log^{1-\beta/H}\left(  x^{-1}\right)  .
\end{align*}
Therefore, for $\beta\geq H=1/d$ (case 4), we see that $v\circ\gamma
^{-1}\left(  x\right)  $ is bounded, and thus we use $\mathrm{K}=1$, while for
$\beta<H=1/d$, $v\circ\gamma^{-1}\left(  x\right)  $ is unbounded, and we thus
use $\mathrm{K}\left(  x\right)  =\log^{1-\beta/H}\left(  x^{-1}\right)  $.
\vspace*{0.1in}

\emph{Step 4: Proof of case 5.} For $dH<1$, let $\varepsilon>0$ such that
$dH+\varepsilon<1$. Using the same argument as in the previous step, and using
the fact that $s^{\varepsilon}\log^{-\beta/H}\left(  1/s\right)  $ is bounded
above by some constant $M$, we compute
\begin{align*}
v\left(  r\right)   &  =\int_{r}^{b-a}s^{-dH}\log^{-\beta/H}\left(
1/s\right)  ds\\
&  =\int_{r}^{b-a}s^{-dH-\varepsilon}s^{\varepsilon}\log^{-\beta/H}\left(
1/s\right)  ds\\
&  \leq M\int_{r}^{b-a}s^{-dH-\varepsilon}ds\\
&  \leq\frac{\left(  b-a\right)  ^{1-dH-\varepsilon}M}{1-dH-\varepsilon}.
\end{align*}
Since this is bounded, so is $v\circ\gamma^{-1}$, and the result of case 5
follows, finishing the proof of the theorem.
\end{proof}

\begin{proof}
[Proof of Corollary \ref{corex1}]The result for $H=1$ is immediate. For $H=0$,
noting that $v\left(  r\right)  =:\int_{r}^{b-a}\log^{-\beta d}\left(
1/s\right)  ds$ is bounded, by Theorem \ref{tcap}, the lower bound holds with
$\varphi=1$.
\end{proof}

\begin{proof}
[Proof of Corollary \ref{corex2}]First note that from the definitions, since a
Hausdorff measure is computed using scalar diameters, and a capacity is
computed using measures supported on the set, these quantities relative to a
subset of $\mathbb{R}$ are invariant when the subset is immersed in
$\mathbb{R}^{d}$.

Therefore, as mentioned in the paragraph following Corollary \ref{corex1}, our
Cantor set $A$ has finite Hausdorff measure $\mathcal{H}_{d-1/H}\left(
A\right)  \leq1$. This is the Hausdorff measure $\mathcal{H}_{\psi}\left(
A\right)  $ with $\psi\left(  x\right)  =x^{d-1/H}$. By the definition of
Hausdorff measure, it is then immediate that, for the function $\varphi\left(
x\right)  =x^{d-1/H}\log^{\beta/H}(1/x)$ with $\beta<0$, the Hausdorff measure
$\mathcal{H}_{\varphi}\left(  A\right)  =0$. Theorem \ref{texa} part 1, upper
bound, implies that $A$ is polar for $B^{H,\beta}$. Now to show that, for
$\beta^{\prime}\geq H$, $B^{H,\beta^{\prime}}$ hits $A$ with positive
probability, by Theorem \ref{texa}, part 1, lower bound, it is sufficient to
show that, with $\tilde{\varphi}(x)=x^{d-\frac{1}{H}}\log^{\beta^{\prime}%
/H}(1/x)$, $\mathcal{C}_{1/\tilde{\varphi}}\left(  A\right)  >0$. Using the
classical criterion mentioned above (see \cite{Beardon:68}), it is sufficient
to show that $\sum_{n=1}^{\infty}2^{-n}\left(  1/\tilde{\varphi}\left(
q^{n}\right)  \right)  <\infty$. Since we chose $q$ such that $q^{1/H-d}=2$,
we compute%
\[
\sum_{n=1}^{\infty}2^{-n}\left(  1/\tilde{\varphi}\left(  q^{n}\right)
\right)  =\sum_{n=1}^{\infty}\log^{-\beta/H}\left(  q^{n}\right)  =\log\left(
1/q\right)  \sum_{n=1}^{\infty}n^{-\beta/H}<\infty.
\]
This finishes the proof of the corollary.
\end{proof}

\begin{proof}
[Proof of Corollary \ref{corex3}]Similarly to the previous proof, thanks to
Theorem \ref{texa}, part 1, lower bound, and thanks to the characterization of
capacity positivity described above, it is sufficient to show that, with
$\tilde{\varphi}(x)=x^{d-\frac{1}{H}}$, $\sum_{n}2^{-n}/\tilde{\varphi}\left(
\prod_{i=1}^{n}q_{i}\right)  <\infty$. By assumption, this series is bounded
above by $\sum_{n}\prod_{i=1}^{n}\left(  1-\frac{c}{i}\right)  $. Using
$\log\left(  1-u\right)  \leq-u$ for all $u>0$, the product in this series is
bounded above as
\begin{align*}
\prod_{i=1}^{n}\left(  1-\frac{c}{i}\right)   &  =\exp\left(  \sum_{i=1}%
^{n}\log\left(  1-c/i\right)  \right)  \leq\exp\left(  -c\sum_{i=1}^{n}%
i^{-1}\right) \\
&  =\exp\left(  -c\log n-c\mathrm{\gamma}_{\mathrm{e}}+o_{n}\left(  1\right)
\right)  =n^{-c}O_{n}\left(  1\right)
\end{align*}
where $\mathrm{\gamma}_{\mathrm{e}}$ is Euler's constant. For any $c>1$, this
is the general term of a converging series, finishing the proof of the
corollary, modulo the statement about the Hausdorff dimension of $\tilde{A}$
which is well known (see \cite{Beardon:68}) and elementary, and can serve as
an exercise for the interested reader.
\end{proof}

\end{document}